\documentclass{amsart}
\usepackage{times,amsmath}

\usepackage{graphicx}
\newtheorem{thm}{Theorem}[section]
\newtheorem{cor}[thm]{Corollary}

\newtheorem{prop}[thm]{Proposition}
\theoremstyle{definition}
\newtheorem{defn}[thm]{Definition}
\theoremstyle{remark}
\newtheorem{rem}[thm]{Remark}

\newtheorem{example}[thm]{Example}
\numberwithin{equation}{section}
\newcommand{\Real}{\mathbb R}

\newcommand{\rn}[1]{\mathbb{R}^{#1}}

\newcommand{\mylabel}[1]{{\label{#1}}}%{{\fbox{#1}}}}

\begin{document}
\title{ Existence of positive definite noncoercive sums of squares in $\Real [x_1,\ldots,x_n]$}%
\author{Gregory C. Verchota}%
\thanks{The author gratefully acknowledges partial support provided by
the National Science Foundation  through award DMS-0401159}
\address{215 Carnegie \\ Syracuse University\\ Syracuse NY 13244}%
\email{gverchot@syr.edu}%
\subjclass{12D15,11E25,35J30,35J40}%
\date{\today}%
\begin{abstract}
Positive definite forms $f\in\mathbb{R}[x_1,\ldots,x_n]$  which are
sums of squares of forms of $\mathbb{R}[x_1,\ldots,x_n]$ are
constructed to have the additional property that the members of any
collection of forms whose squares sum to $f$ must share a nontrivial
complex root in $\mathbb{C}^n$.
\end{abstract}
\maketitle
\section{Introduction}\mylabel{Introduction}
Let $f\in \Real [x_1,\ldots,x_n]$ be a \textit{form}, i.e.
homogeneous polynomial.  Suppose $f$ is a \textit{sum of squares}
($sos$) of  forms in $\Real [x_1,\ldots,x_n]$ and is
\textit{positive definite} ($pd$),
  $f(\mathbf{a})>0$ for all $ \mathbf{a}\in\rn n\setminus
  \{\mathbf{0}\}$.
 Writing $f=\sum p_j^2$ \; this is equivalent to saying that the forms $p_j$
 share no common nontrivial real root from $\rn n$.

 \begin{multline}\mylabel{question}
 \mbox{\textit{Suppose a \emph{positive definite} form} $f$ \textit{has at least one} sos \textit{representation}.
  \textit{Does} $f$ \textit{necessar-}}\\\mbox{\textit{ ily have a representation} $f=\sum q_k^2$ }\mbox{\textit{with}
  $q_k\in \mathbb{R}[x_1,\ldots,x_n]\;\;  $\textit{and the} $q_k$ \textit{sharing no}\hskip.5in}\\\mbox{\textit{  common complex root from} $\mathbb{C}^n\setminus
  \{\mathbf{0}\}$?\hskip2.9in}
 \end{multline}

For example,\vskip.1in

\noindent (i) the \textit{positive semi-definite} ($psd$)
$x_1^2=p^2\in\mathbb{R}[x_1,x_2,x_3]$ is uniquely represented as an
$sos$, and $p(0,1,i)=0$;\vskip.1in

\noindent (ii)  $x_1^2+x_2^2\in\mathbb{R}[x_1,x_2]$ is $pd$ with
$x_1$ and $x_2$ sharing no common nontrivial complex root;\vskip.1in

\noindent (iii) $f=(x_1^2+x_2^2)^2=p^2$ is $pd$ with the quadratic
form $p$ having the root $(1,i)\in\mathbb{C}^2$.  But also $f=
(x_1^2)^2+(\sqrt{2}x_1x_2)^2+(x_2^2)^2$ or
$(x_1^2-x_2^2)^2+(2x_1x_2)^2$ and in each case the quadratic forms
now share no common nontrivial complex root.\vskip.1in

Though not the subject of this article, the study of boundary value
problems for elliptic partial differential equations (PDE) motivates
question \eqref{question}. Denote by
$\partial=(\partial_1,\ldots,\partial_n)=(\frac{\partial}{\partial
x_1},\ldots,\frac{\partial}{\partial x_n})$ the vector of first
partial derivatives for $\mathbb{R}^n$.  Let
$\alpha\in\mathbb{N}^n_0$ denote a multi-index.  Define
$|\alpha|=\alpha_1+\cdots+\alpha_n$ and
$\partial^\alpha=\partial_1^{\alpha_1}\cdots\partial_n^{\alpha_n}$.

A theorem of N. Aronszajn and K. T. Smith \cite{Agm65} may be stated
as

\textit{Let} $p_1,\ldots,p_r\in\mathbb{R}[x_1,\ldots,x_n]$
\textit{be forms of degree} $d$.  \textit{Let}
$\Omega\subset\mathbb{R}^n$ \textit{be a bounded open connected set
with suitably regular boundary and let} $\overline{\Omega}$
\textit{be its closure.} \textit{Then the} integro-differential
quadratic form
\begin{equation}\mylabel{intform}
\sum_j\int_\Omega|p_j(\partial)u|^2dx
\end{equation}
\textit{is} coercive \textit{over all functions} $u$ \textit{which
have continuous partial derivatives of order} $d$ \textit{in}
$\Omega$ \textit{that extend continuously to} $\overline{\Omega}$ if
and only if \textit{the system}
$$p_1=p_2=\cdots=p_r=0$$
\textit{has no solution} $\mathbf{a}\in
\mathbb{C}^n\setminus\{\mathbf{0}\}$.\vskip.1in

For \eqref{intform} to be \textit{coercive} over the collection of
functions $u$ it is required, by definition, that there be constants
$C>0$ and $c_0\in \mathbb{R}$ independent of the functions $u$ so
that
\begin{equation}\mylabel{pdecoercive}
\sum_j\int_\Omega|p_j(\partial)u|^2dx\geq
C\int_\Omega\sum_{|\alpha|\leq d} |\partial^\alpha u|^2
dx-c_0\int_\Omega | u|^2 dx
\end{equation}
for all $u$ in the collection.  Once this estimate is obtained
various elliptic boundary value problems can be solved.

The Aronszajn-Smith theorem gives a precise algebraic
characterization of all integro-differential forms \eqref{intform}
for which the coercive estimate \eqref{pdecoercive} can hold.  The
integro-differential forms \eqref{intform} are termed
\textit{formally positive} because of their \textit{sos} shape.  S.
Agmon \cite{Agm58} improved this result by proving a necessary and
sufficient (and more complicated) algebraic condition on all
integro-differential forms
\begin{equation}\mylabel{intdiffform}
Re\sum_{|\alpha|\leq d}\sum_{|\beta|\leq d}\int_\Omega a_{\alpha
\beta}\partial^\alpha u \overline{\partial^\beta u} dx
\end{equation}
not only the formally positive, that give rise to self-adjoint
linear properly elliptic differential operators
\begin{equation}\mylabel{diffop}
L(\partial)=\sum_{|\alpha|\leq d}\sum_{|\beta|\leq d}a_{\alpha
\beta}\partial^{\alpha +\beta}
\end{equation}
and their regular boundary value problems \cite{Agm58}\cite{Agm60}.
When $a_{\alpha\beta}\in \mathbb{R}$ and the integro-differential
form is formally positive, $L$ corresponds to a polynomial $f$  of
degree $2d$ that is a sum of squares.

With his algebraic characterization Agmon solved completely the
\textit{coerciveness problem for integro-differential forms} in the
theory of linear PDE.  However, the \textit{coerciveness problem for
linear differential operators} $L(\partial)=\sum_{|\alpha|\leq
2d}a_\alpha
\partial^\alpha$ has not been solved.  This problem can be stated
in a way that leads back to the question about sums of squares in
$\mathbb{R}[x_1,\ldots,x_n]$.

Instead of the integro-differential form one begins with the
homogeneous constant coefficient operator in $\mathbb{R}^n$
$$L(\partial)=\sum_{|\alpha|=
2d}a_\alpha
\partial^\alpha$$ $a_\alpha\in\mathbb{R}$.  These will be
self-adjoint.  Suppose $L$ is \textit{elliptic} (equivalent to
properly elliptic in this setting) $L(\xi)>0$ for all $\xi\in
\mathbb{R}^n\setminus\{\mathbf{0}\}$.  In general $L$ can be
rewritten an infinity of ways in the shape \eqref{diffop}
\begin{equation}\mylabel{homdiffop}
L(\partial)=\sum_{|\alpha|=|\beta|= d}a_{\alpha
\beta}\partial^{\alpha +\beta}
\end{equation}
and therefore admits an infinity of integro-differential forms
\eqref{intdiffform}.  Is there any choice of rewriting
\eqref{homdiffop} that yields a coercive estimate?

This fundamental question is broader than what can be answered here.
Instead the question will be specialized to the setting of the
Aronszajn-Smith theorem.

Suppose it is further known that the homogeneous differential
operator is an
 \textit{sos}, $L(\partial)=\sum p_j^2(\partial)$.
Then the theorem provides the necessary and sufficient algebraic
condition for the integro-differential form \eqref{intform} to be
coercive \eqref{pdecoercive}.  If the form were to fail the
algebraic condition and thus fail to be coercive is there another
way to write the differential operator $L$ \textit{as a sum of
squares} and thereby use the theorem again to obtain the coercive
estimate for a new integro-differential form associated to $L$ and
thus  solve boundary value problems for $L$?  This is question
\eqref{question}.

All the results and proofs of this article are independent of these
 PDE considerations.  Some more will be said about PDE in the last
 section.

 \begin{defn}
$f\in\mathbb{R}[x_1,\ldots,x_n]$ is called a \emph{sum of squares}
(an \emph{sos}) if there exist polynomials
$p_1,\ldots,p_r\in\mathbb{R}[x_1,\ldots,x_n]$ so that $f$ has the
representation $f=\sum_{j=1}^r p_j^2$

 \end{defn}

\begin{defn}
An $sos$ $f\in\mathbb{R}[x_1,\ldots,x_n]$ is called \emph{coercive}
or a \emph{coercive sum of squares} if there exists a representation
\begin{equation}\mylabel{sos}
f=\sum_{j=1}^r p_j^2
\end{equation}
with $p_1,\ldots,p_r\in\mathbb{R}[x_1,\ldots,x_n]$ such that there
are \textit{no} solutions $\mathbf{a}\in
\mathbb{C}^n\setminus\{\mathbf{0}\}$ to the system
\begin{equation}\mylabel{system}
p_1=\cdots=p_r=0
\end{equation}
When such an $f$ is homogeneous it is also called a \emph{coercive
form}.
 \end{defn}

 To be clear
 \begin{defn}
An $sos$ $f\in\mathbb{R}[x_1,\ldots,x_n]$ is called
\emph{noncoercive} or a \emph{noncoercive sos} if there exists a
representation \eqref{sos} for $f$ and if \emph{every} such
representation has a nontrivial solution in $\mathbb{C}^n$ to the
\emph{corresponding system} \eqref{system}.
 \end{defn}

Question \eqref{question} asks if every positive definite $sos$ is
coercive.  The aim of this article is to establish, by construction,
the existence of
 positive definite noncoercive sums of squares. That this
can be done is related to the well known fact that \textit{not every
positive definite polynomial is a sum of squares}.

If every $pd$ polynomial were an $sos$ the answer to question
\eqref{question} would be \textit{yes}. This follows because
positive definiteness of $f$ allows
\begin{equation}\mylabel{bracket}
f=[f-\epsilon (x_1^{2d}+\cdots+x_n^{2d})]+\epsilon (x_1^{2d}+\cdots)
\end{equation}
with the bracketed term $pd$ for $\epsilon>0$ small enough.  When
the bracketed term is an $sos$, \eqref{bracket} is an $sos$
representation for $f$ that satisfies the definition of coercive
$sos$.

We adopt standard notations for $psd$ homogeneous polynomials
\cite{CL77}\cite{BCR98} p.111.  $P_{n,d}$ denotes the set of
$f\in\mathbb{R}[x_1,\ldots,x_n]$ homogeneous of degree $d$ that are
\textit{nonnegative} on $\mathbb{R}^n$.  $\Sigma_{n,d}$ denotes the
set of all $f\in P_{n,d}$ that are $sos$.  These sets are nonempty
only when $d$ is an even number.

\textsc{For the remainder of this article all polynomials will be
homogeneous polynomials, or forms.}

(Homogenization can be used for other statements.)

The argument given above together with Hilbert's results on positive
polynomials that are $sos$ \cite{Hil88}, \cite{Rez07} immediately
yields the Theorem
\begin{multline}\mylabel{positiveresult}
\mbox{\textit{If} $n\leq 2$ \textit{and} $d$ \textit{is an even
natural number, or if} $d=2$ \textit{and} $n$ \textit{is a natural
number, or if}}\\\mbox{ $(n,d)=(3,4)$\textit{, then every} $pd$
\textit{form of} $P_{n,d}$ \textit{is a coercive sum of
squares.}\hskip2in}
\end{multline}
The result of Hilbert \cite{Raj93}, \cite{Swa00}, \cite{Rud00},
\cite{Pfi04}, \cite{PR00} used here is that $P_{3,4}=\Sigma_{3,4}$,
while $P_{2,2p}=\Sigma_{2,2p}$ and $P_{n,2}=\Sigma_{n,2}$ are
elementary. See \cite{BCR98} pp.111-112.

Hilbert further proved that in every other  case $\Sigma_{n,2p}$ is
a proper subset of $P_{n,2p}$, eliminating the argument based on
\eqref{bracket}.  It was T. S. Motzkin \cite{Mot67} who first
published explicit examples of positive semi-definite polynomials
that were not $sos$. There are now various examples of these, e.g.
\cite{Rob73},\cite{CL77},\cite{CL77b},\cite{LL78}; see \cite{Rez00}
for more. We found two of these to be very useful for the purpose
here. Both are of Motzkin type and due to M. D. Choi and T. Y. Lam.

\begin{equation}\mylabel{q}
q(w,x,y,z)=w^4+x^2y^2+y^2z^2+z^2x^2-4wxyz
\end{equation}
and
$$s(x,y,z)=x^4y^2+y^4z^2+z^4x^2-3x^2y^2z^2$$
Both are nonnegative ($psd$) by the arithmetic-geometric mean
inequality and neither is an $sos$.  Thus $q\in
P_{4,4}\setminus\Sigma_{4,4}$ and $s\in
P_{3,6}\setminus\Sigma_{3,6}$.

For $\eta\geq 0$ define
\begin{equation}\mylabel{qeta}
q_\eta=q+\eta(x^4+y^4+z^4)
\end{equation}
$$s_\eta=s+\eta(x^6+y^6+z^6)$$
For $\eta>0$,  $q_\eta$ and $s_\eta$ are $pd$.  As long as $\eta$ is
small enough each is not an $sos$.  This follows by an elementary
topological argument first given by R. M. Robinson \cite{Rob73}
pp.267-268 which, moreover, shows the sets $\Sigma$ to be
topologically  closed sets.  It is also true that for all $\eta$
large enough $q_\eta$ and $s_\eta$ are $sos$.  See, for example,
p.269 of \cite{Rob73} (in the case of  $q_\eta$ it can be verified
that the $w^4$ term in $q$ obviates the need to add $\eta w^4$).
Consequently for each polynomial there is a smallest value of
$\eta$, $\eta_0>0$, that makes $q_\eta$ or $s_\eta$ $sos$ (cf. also
the proof of Corollary 5.6 \cite{CLR95} p.122).  In Section
\ref{quartic} it is shown for the quartic $q$ that the square root
of this value is the smallest positive root of
$X^3-\frac{1}{2}X+\frac{1}{9}=0$, and that
\begin{multline}\mylabel{sosquartic}
q_{\eta_0}(w,x,y,z)=(w^2-\sqrt{\eta_0}(x^2+y^2+z^2))^2+\\\dfrac{2}{9\sqrt{\eta_0}}[(3\sqrt{\eta_0}wx-yz)^2+(3\sqrt{\eta_0}wy-zx)^2+(3\sqrt{\eta_0}wz-xy)^2]
\end{multline}
In addition, it is proved that there is exactly one Gram matrix (or
Gramian \cite{Gel89}) that represents the polynomial $q_{\eta_0}$.
This means that every other $sos$ representation for $q_{\eta_0}$ is
merely a sum of squares of quadratics that are linear combinations
of the quadratics of \eqref{sosquartic}.  Thus any common complex
roots must be the same among all representations.

The Gram matrix method of Choi, Lam and B. Reznick \cite{CLR95},
used for studying $sos$ representations of polynomials, is put into
a tensor setting in Section \ref{setup}.  Every form  of degree $2p$
is nonuniquely represented  by a symmetric matrix (rank-2 symmetric
tensor) acting as a quadratic form on the vector space of rank-p
symmetric tensors.  These are termed \textit{representation
matrices} for the form.  The Gram matrices are those representation
matrices that are $psd$, necessary and sufficient for an $sos$
representation.

The polynomial \eqref{sosquartic} provides an example of a positive
definite quartic with a unique Gram matrix.  A positive definite
sextic with a unique Gram matrix has  previously been identified by
Reznick in \cite{Pra06}.  It is like the ones that will be
constructed in Section \ref{cubic} from the $s_\eta$.

However wonderful it is, $q_{\eta_0}$ is coercive.  It is proved in
Section \ref{proof1} that
\begin{equation}\mylabel{noncoercsos}
(u^2+v^2+vw)^2+q_{\eta_0}(w,x,y,z)
\end{equation}
is \textit{positive definite} and \textit{noncoercive} in
$\Sigma_{6,4}$.  In effect the uniqueness of representation of
\eqref{sosquartic} and the presence of the monomial $vw$ forces a
uniqueness of representation upon \eqref{noncoercsos}, while
$(1,i,0,0,0,0)$ is a solution to the corresponding system of
quadratic equations \eqref{system}.  It follows from the definition
of coercive $sos$ that any form $f\in\mathbb{R}[x_1,\ldots,x_n]$ of
even degree $d$ such that $f+x_{n+1}^d$ is a coercive $sos$ must
itself be a coercive $sos$. Consequently monomials $x_7^4, x_8^4,
\ldots$ can be added to \eqref{noncoercsos} preserving all required
properties and the following theorem and partial answer to question
\eqref{question} is obtained.

\begin{thm}\mylabel{thm1}
For $n\geq 6$, $\Sigma_{n,4}$ contains polynomials that are positive
definite and noncoercive.
\end{thm}

Theorem \ref{thm1} is really a statement about certain cones of
polynomials. After a scaling \eqref{sosquartic} can be rewritten
\begin{multline}\mylabel{conerep}
a_1(x_1^2-\gamma(x_2^2+x_3^2+x_4^2))^2+a_2(x_1x_2-x_3x_4)^2+a_3(x_1x_3-x_4x_2)^2+a_4(x_1x_4-x_2x_3)^2
\end{multline}
where it happens that for all values of $\gamma$, $0<\gamma
<\frac{1}{3}$ and all positive $a_1,\ldots, a_4$, the forms
\eqref{conerep} are $pd$ with a unique Gram matrices.

\begin{cor}
For $n\geq 6$ there exist nonempty collections of quadratic forms
$\{p_1,\ldots,p_r\}\subset\mathbb{R}[x_1,\ldots,x_n]$ so that there
exist no nontrivial solutions from $\mathbb{R}^n$ to the systems
$p_1=p_2=\cdots=p_r=0$, and so that every $f=\sum a_jp_j^2$, with
positive coefficients $a_1,\ldots,a_r$, is a noncoercive $sos$.
\end{cor}

The Choi-Lam sextic form $s$ \eqref{q} possesses more structure than
its quartic counterpart $q$.  First it is an \textit{even} form.  A
form $f$ is \textit{even} if it is also a polynomial in
$x_1^2,x_2^2,\ldots,x_n^2$. Second it is \textit{symmetric}.  A form
$f$ is \textit{symmetric} if for every permutation $\sigma$ on $n$
objects $f(\mathbf{x})=f(\sigma(\mathbf{x}))$.  The construction
\eqref{qeta} of the forms $s_\eta$ preserves both of these
properties.  In Section  \ref{cubic}, for $s_\eta(x,y,z)$ with a
unique Gram matrix, it is proved that when $x^2$ is replaced with
$w^2+x^2$ the resulting form is $pd$ and noncoercive.

\begin{thm}\mylabel{thm2}
For $n\geq 4$, $\Sigma_{n,6}$ contains polynomials that are positive
definite and noncoercive.
\end{thm}

The additional structure provided by the non-$sos$  $s$ seems to be
the reason Theorem \ref{thm2} comes closer than Theorem \ref{thm1}
to being a  complete result.  As remarked on p.263 of \cite{Rez00}
and in \cite{Har99}, in any dimension every $psd$ \textit{even
symmetric quartic} form is an $sos$.  Further, the replacement of
$x^2$ with $w^2+x^2$ that works in the sextic construction seems to
rely more on the \textit{even} property than it does on symmetry. It
turns out that every $psd$ \textit{even quartic} form in $n=4$ or
fewer variables is a sum of squares. This follows from results of P.
H. Diananda \cite{Dia62}.  Thus constructing a quartic noncoercive
$sos$ for $n=5$ from an even form in $4$ variables in a way
analogous to the sextic case is not possible.  On the other hand the
Horn form \cite{HN63} pp. 334-335 \cite{Dia62} p.25 \cite{Rez00}
p.260 provides a $psd$ \textit{even quartic} form for $n=5$ that is
not an $sos$. See \cite{CL77} pp.394-396.

Between the coercive Theorem \eqref{positiveresult} and the
noncoercive Theorems \ref{thm1} and \ref{thm2}, dimensions $4$ and
$5$ for the former and $3$ for the latter remain obscure.  This
puzzle will be discussed further in Section \ref{game}.

\section{A multilinear setup}\mylabel{setup}

At first let $\mathbf{e}^1,\ldots,\mathbf{e}^n$ and
$\mathbf{e}_1,\ldots,\mathbf{e}_n$ be the standard (contravariant
and covariant) basis vectors for $\mathbb{R}^n$.  The scalar product
of vector and covector is denoted $\mathbf{x}\cdot \mathbf{u}=\sum
x_j u^j$ where $x_1,\ldots, u^1,\ldots$ are the standard coordinates
of $\mathbf{x}$ and $\mathbf{u}$.  The nonnegative integers are
denoted $\mathbb{N}_0$.  For a multi-index $\alpha\in \mathbb{N}_0$
its \textit{order} is $|\alpha|=\alpha +\cdots +\alpha_n$, and
$\alpha!=\alpha_1!\cdots \alpha_n!$  For $\mathbf{x}\in
\mathbb{R}^n$, $\mathbf{x}^\alpha=x_1^{\alpha_1}\cdots
x_n^{\alpha_n}$.

The (contravariant) \textit{tensors} $\mathbf{t}$ of \textit{rank}
$p$ are multilinear ($p$-linear) forms mapping $p$ vectors of
$\mathbb{R}^n$ to $\mathbb{R}$ by $$\mathbf{t}\cdot
\mathbf{x}^1\mathbf{x}^2\cdots \mathbf{x}^p=\sum
x^1_{j_1}x^2_{j_2}\cdots x^p_{j_p} t^{j_1 j_2\cdots j_p}$$ The
\textit{coordinates} of $\mathbf{t}$ are $t^{j_1 \cdots j_p}$ and
are  obtained by $\mathbf{t}\cdot \mathbf{e}^{j_1}\cdots
\mathbf{e}^{j_p}=t^{j_1 \cdots j_p}$.  See \cite{Van70} pp.74-75,
80-81.

Given $p$ (co)vectors $\mathbf{u}_1,\ldots,\mathbf{u}_p$ a tensor
$\mathbf{t}$  of rank $p$ may be defined by the \textit{tensor
product} $$\mathbf{t}=\mathbf{u}_1\otimes\cdots\otimes\mathbf{u}_p$$
which acts multilinearly as
\begin{equation}\mylabel{tensorproductact}
\mathbf{t}\cdot \mathbf{x}^1\mathbf{x}^2\cdots
\mathbf{x}^p=(\mathbf{x}^1\cdot \mathbf{u}_1)(\mathbf{x}^2\cdot
\mathbf{u}_2)\cdots(\mathbf{x}^p\cdot \mathbf{u}_p)
\end{equation}
so that $t^{j_1 \cdots j_p}=u_1^{j_1}\cdots u_p^{j_p}$.

The collection of tensors of rank $p$, $T^p(\mathbb{R}^n)$, forms a
vector space over $\mathbb{R}$ of dimension $n^p$ with standard
basis $$\{\mathbf{e}_{j_1}\otimes\cdots\otimes\mathbf{e}_{j_p} :
1\leq j_\nu \leq n\}$$ Let $\mathfrak{S}_p$ denote the symmetric
group of all permutations of $p$ objects.  For each
$\sigma\in\mathfrak{S}_p$ the map
$$P_\sigma(\mathbf{e}_{j_1}\otimes\cdots\otimes\mathbf{e}_{j_p})=\mathbf{e}_{j_{\sigma(1)}}\otimes\cdots\otimes\mathbf{e}_{j_{\sigma(p)}}$$
defines a permutation of the basis vectors of $T^p(\mathbb{R}^n)$
and thereby induces a (unique) linear isomorphism on
$T^p(\mathbb{R}^n)$ \cite{Yok92} p.43.  If
$P_\sigma(\mathbf{t})=\mathbf{t}$ for all $\sigma\in\mathfrak{S}_p$,
then $\mathbf{t}$ is called a \textit{symmetric tensor}.  The set of
all symmetric tensors of rank $p$, $S^p(\mathbb{R}^n)$, also forms a
vector space over $\mathbb{R}$.  The linear operator $$
Sym=Sym_p=\dfrac{1}{p!}\sum_{\sigma\in\mathfrak{S}_p} P_\sigma$$ is
a projection from $T^p(\mathbb{R}^n)$ \textit{onto}
$S^p(\mathbb{R}^n)$ so that
\begin{equation}\mylabel{spbasis}
\{Sym(\mathbf{e}_{j_1}\otimes\cdots\otimes\mathbf{e}_{j_p}): 1\leq
j_1\leq\cdots \leq j_p \leq n\}
\end{equation}
forms a basis for $S^p(\mathbb{R}^n)$.  Further,
\begin{equation}\mylabel{dimsp}dim(S^p(\mathbb{R}^n))=\left(\begin{array}{c}n+p-1\\p\end{array}\right)\end{equation}
\cite{Yok92} pp.47-48.

Given indices $j_1\leq\cdots\leq j_p$ as in \eqref{spbasis} let
$\alpha_k$ equal the number of indices equal to $k$ for $1\leq k\leq
n$. In this way the multi-indices $\alpha\in \mathbb{N}^n_0$ of
order $p$ are put in one-to-one correspondence with the basis
elements of $S^p(\mathbb{R}^n)$.  Denote
\begin{equation}\mylabel{E}
\mathbf{E}_\alpha=Sym(\mathbf{e}_{j_1}\otimes\cdots\otimes\mathbf{e}_{j_p})
\end{equation}
for each basis element in \eqref{spbasis} where $\alpha$ corresponds
to $j_1\leq\cdots\leq j_p$.  $$\hbox{\textit{The coordinates of} }
\mathbf{E}_\alpha\hbox{ \textit{(as a tensor in}
}T^p(\mathbb{R}^n)\hbox{\textit{)} }E^{k_1\cdots k_p}_\alpha\hbox{
\textit{are either} $0$ \textit{or} $\frac{\alpha!}{p!}$\textit{,
and sum to} $1$.}$$
\begin{example}
(i) For $p=2$,
$\mathbf{E}_{(2,0\ldots,0)}=\mathbf{e}_1\otimes\mathbf{e}_1$ with
$E^{11}_{(2,0\ldots,0)}=1$ the only nonzero coordinate.

\noindent$\mathbf{E}_{(1,1,0\ldots,0)}=\frac{1}{2}(\mathbf{e}_1\otimes\mathbf{e}_2+\mathbf{e}_2\otimes\mathbf{e}_1)$
with $E^{12}_{(1,1,0\ldots,0)}=E^{21}_{(1,1,0\ldots,0)}=\frac{1}{2}$
the only nonzero coordinates.

Thus $\{\mathbf{E}_\alpha : |\alpha|=2\}$ is identified with an
orthogonal basis for the $n\times n$ symmetric matrices under the
Hilbert-Schmidt inner product.

\noindent(ii) For $p=3$,
$\mathbf{E}_{(3,0\ldots,0)}=\mathbf{e}_1\otimes\mathbf{e}_1\otimes\mathbf{e}_1$.

\noindent$\mathbf{E}_{(2,1,0\ldots,0)}=\frac{1}{3}(\mathbf{e}_1\otimes\mathbf{e}_1\otimes\mathbf{e}_2+\mathbf{e}_1\otimes\mathbf{e}_2\otimes\mathbf{e}_1+
\mathbf{e}_2\otimes\mathbf{e}_1\otimes\mathbf{e}_1)$.

\noindent$\mathbf{E}_{(1,1,1,0\ldots,0)}=\frac{1}{6}(\mathbf{e}_1\otimes\mathbf{e}_2\otimes\mathbf{e}_3+\mathbf{e}_3\otimes\mathbf{e}_1\otimes\mathbf{e}_2+
\mathbf{e}_2\otimes\mathbf{e}_3\otimes\mathbf{e}_1+\mathbf{e}_1\otimes\mathbf{e}_3\otimes\mathbf{e}_2+\mathbf{e}_2\otimes\mathbf{e}_1\otimes\mathbf{e}_3+
\mathbf{e}_3\otimes\mathbf{e}_2\otimes\mathbf{e}_1)$.
\end{example}

For a vector $\mathbf{x}\in\mathbb{R}^n$ (or $\mathbb{C}^n$) each
basis element $\mathbf{E}_\alpha \in S^p(\mathbb{R}^n)$ therefore
acts multilinearly on $\mathbf{x}$ as
\begin{equation}\mylabel{actsmulti}
\mathbf{E}_\alpha\cdot\mathbf{x}\mathbf{x}\cdots\mathbf{x}=\mathbf{x}^\alpha
\end{equation}
Therefore

\vskip.25in  \noindent \textit{The vector space} $S^p(\mathbb{R}^n)$
\textit{is  isomorphic to the vector space of homogeneous
polynomials of degree} $p$ \textit{from}
$\mathbb{R}[x_1,\ldots,x_n]$.

\vskip.25in See, for example, Theorem 2.5 p.67 of \cite{Yok92}.

In the same way the vector space  of (covariant)  tensors
$T_p(\mathbb{R}^n)$ dual to $T^p(\mathbb{R}^n)$ (\cite{Yok92},
pp.53-54) is formed.  Putting
$\mathbf{s}=\mathbf{x}^1\otimes\cdots\otimes\mathbf{x}^p\in
T_p(\mathbb{R}^n)$, \eqref{tensorproductact} can be rewritten as the
dual pairing
\begin{equation}\mylabel{dualpairing}
\mathbf{t}\cdot \mathbf{s}=(\mathbf{x}^1\cdot
\mathbf{u}_1)(\mathbf{x}^2\cdot
\mathbf{u}_2)\cdots(\mathbf{x}^p\cdot \mathbf{u}_p)
\end{equation}
A basis for the (covariant) symmetric tensors $S_p(\mathbb{R}^n)$ is
defined similarly to \eqref{spbasis}, and basis elements
$\mathbf{E}^\alpha$, $|\alpha|=p$, are defined as in \eqref{E}.  By
the normalizations
\begin{equation}\mylabel{normalbasis}
\mathbf{N}_\alpha=\sqrt{\dfrac{p!}{\alpha!}}\mathbf{E}_\alpha\mbox{
and }\mathbf{N}^\alpha=\sqrt{\dfrac{p!}{\alpha!}}\mathbf{E}^\alpha
\end{equation}
one obtains dual bases
\begin{equation}\mylabel{ndotn}
\mathbf{N}_\alpha\cdot\mathbf{N}^\beta=\delta_\alpha^\beta
\end{equation}
where the Dirac delta is equal to $0$ when $\alpha\neq\beta$ and $1$
otherwise.

Because these dual symmetric  spaces are isomorphic,  no longer will
any distinction  be made between them.  Instead $S^p(\mathbb{R}^n)$
will be considered an inner product space with inner product formed
as in \eqref{dualpairing}.  Bases will be written
$\{\mathbf{E}_\alpha:|\alpha|=p\}$,
$\{\mathbf{N}_\alpha:|\alpha|=p\}$  an orthogonal and an orthonormal
basis respectively.  Vectors  of $\mathbb{R}^n$ will be enumerated
$\mathbf{x}^1,\mathbf{x}^2,\ldots,\mathbf{u}^1,\ldots$ with
subscripts indicating coordinates $\mathbf{x}=(x_1,x_2,\ldots)$,
$\mathbf{x}^1=(x_1^1,x_2^1,\ldots),\ldots$

A convenient notation for the tensor product of $p$
\textit{identical} vectors is
\begin{equation}\mylabel{xtothep}
\mathbf{x}^{\otimes p}=\mathbf{x}\otimes\cdots\otimes\mathbf{x}\in
S^p(\mathbb{R}^n)
\end{equation}
When $\mathbf{x}\neq \mathbf{0}$ the tensor $\mathbf{x}^{\otimes p}$
will be referred to as  a \textit{rank-one tensor} even though it is
an element of $S^p(\mathbb{R}^n)$.  For example, when $p=2$ all
$n\times n$ symmetric matrices that have rank $1$ are given by
$\mathbf{x}^{\otimes 2}=\mathbf{x}\otimes\mathbf{x}$.  Now
\eqref{actsmulti} becomes
$$\mathbf{E}_\alpha\cdot\mathbf{x}^{\otimes p}=\mathbf{x}^\alpha\mbox{,  $|\alpha|=p$.}$$

Since $S^p(\mathbb{R}^n)$ is a real vector space, the foregoing can
be done with it in place of $\mathbb{R}^n$.  Of particular interest
is the space $S^2(S^p(\mathbb{R}^n))$ isomorphic to the space of
$\left(\begin{array}{c}n+p-1\\p\end{array}\right)\times
\left(\begin{array}{c}n+p-1\\p\end{array}\right)$ real symmetric
matrices.  These matrices will be referred to below as the
\textit{representation matrices}.

Given any $\mathbf{t}\in  S^p(\mathbb{R}^n)$ the notation of
\eqref{xtothep} will be applied as $\mathbf{t}^{\otimes
2}=\mathbf{t}\otimes\mathbf{t}\in S^2(S^p(\mathbb{R}^n))$.  Given
also $\mathbf{s}$, we introduce the notation
$$\mathbf{s}\otimes_s\mathbf{t}=\mathbf{s}\otimes\mathbf{t}+\mathbf{t}\otimes\mathbf{s}$$
noting that
$$\mathbf{t}\otimes_s\mathbf{t}=2 \mathbf{t}\otimes\mathbf{t}$$
and $$(\mathbf{s}+\mathbf{t})^{\otimes 2}=\mathbf{s}^{\otimes
2}+\mathbf{s}\otimes_s\mathbf{t}+\mathbf{t}^{\otimes 2}$$

A basis for the vector space $S^2(S^p(\mathbb{R}^n))$ is
 \begin{equation}\mylabel{basiss2}
\{\mathbf{E}_\alpha\otimes_s\mathbf{E}_\beta: |\alpha|=|\beta|=p\}
 \end{equation}
 It contains
 $\left(\begin{array}{c}\left(\begin{array}{c}n+p-1\\p\end{array}\right)+1\\\\2\end{array}\right)$
 elements.
More general elements of $S^2(S^p(\mathbb{R}^n))$ will be denoted in
script as with
 $\mathcal{S}$ or $\mathcal{G}$.  All act as symmetric bilinear
 (quadratic) forms on $S^p(\mathbb{R}^n)$
 $$\mathcal{S}\cdot\mathbf{s}\mathbf{t}=\mathcal{S}\cdot\mathbf{t}\mathbf{s}$$
 For example
 $$\dfrac{p!p!}{\alpha!\beta!}\mathbf{E}_\alpha\otimes_s\mathbf{E}_\beta\cdot
 \mathbf{E}_\psi\mathbf{E}_\omega=\delta_\alpha^\psi\delta_\beta^\omega+\delta_\alpha^\omega\delta_\beta^\psi=0\mbox{, $1$ or $2$}$$
 and in particular
 \begin{equation}\mylabel{xa+b}
\dfrac{1}{2}\mathbf{E}_\alpha\otimes_s\mathbf{E}_\beta\cdot\mathbf{x}^{\otimes
p}\mathbf{x}^{\otimes p}=\mathbf{x}^{\alpha+\beta}
 \end{equation}

 By choosing a linear ordering for the multi-indices of order $p$, an
 isomorphism of $S^2(S^p(\mathbb{R}^n))$ and the $\left(\begin{array}{c}n+p-1\\p\end{array}\right)\times
\left(\begin{array}{c}n+p-1\\p\end{array}\right)$ symmetric matrices
can be made explicit.  Given \eqref{xa+b} the one that is apparently
most computationally convenient is induced by the mapping
\begin{equation}\mylabel{isomorphism}
\mathbf{E}_\alpha\otimes_s\mathbf{E}_\beta\mapsto\left(\delta_\alpha^\psi\delta_\beta^\omega+\delta_\alpha^\omega\delta_\beta^\psi\right)_{|\psi|=|\omega|=p}
\end{equation}
In this way an element of $S^2(S^p(\mathbb{R}^n))$ is assigned a
\textit{representation matrix} and \textit{vice versa}.  For
example, with linear order $\alpha\prec\beta\prec\cdots$, the tensor
$(a_\alpha\mathbf{E}_\alpha+a_\beta\mathbf{E}_\beta+\cdots)^{\otimes
2}=a_\alpha^2\mathbf{E}_\alpha^{\otimes 2}+a_\alpha
a_\beta\mathbf{E}_\alpha\otimes_s \mathbf{E}_\beta +\cdots$ is
assigned the matrix $\left(\begin{array}{llc}a_\alpha^2&a_\alpha
a_\beta&\cdots\\a_\alpha a_\beta&a_\beta^2&\cdots
\\\vdots& \vdots & \end{array}\right)$, and by \eqref{xa+b}
represents the form $a_\alpha^2\mathbf{x}^{2\alpha}+2 a_\alpha
a_\beta\mathbf{x}^{\alpha+\beta}
+\cdots=(a_\alpha\mathbf{x}^\alpha+a_\beta\mathbf{x}^\beta+\cdots)^{
2}$.\vskip.1in

\textit{A tensor of} $S^2(S^p(\mathbb{R}^n))$ \textit{and its
representation matrix will be denoted by the same symbol.}\vskip.1in

In addition \eqref{xa+b} shows that\vskip.1in

\noindent \textit{Every element of} $S^2(S^p(\mathbb{R}^n))$
represents \textit{ a homogeneous polynomial in}
$\mathbb{R}[x_1,\ldots,x_n]$ \textit{of degree} $2p$, \textit{and
every such homogeneous polynomial} can be represented \textit{by an
element of} $S^2(S^p(\mathbb{R}^n))$.\vskip.1in

Such representations are not unique.  $S^2(S^p(\mathbb{R}^n))$ is
not isomorphic to $S^{2p}(\mathbb{R}^n)$.  The respective dimensions
are related by
\begin{equation}\mylabel{numbers}
\left(\begin{array}{c}\left(\begin{array}{c}n+p-1\\p\end{array}\right)+1\\\\2\end{array}\right)>\left(\begin{array}{c}n+2p-1\\2p\end{array}\right)
\end{equation}

The following can be found on p.109 of \cite{CLR95}.\vskip.1in

\noindent\textit{The subspace}
\begin{equation}\mylabel{deltas}
A^{2,p,n}=\{\Delta\in S^2(S^p(\mathbb{R}^n)):
\Delta\cdot\mathbf{x}^{\otimes p}\mathbf{x}^{\otimes p}=0 \mbox{ for
every }\mathbf{x}\in \mathbb{R}^n\}
\end{equation}
\textit{has as its dimension the difference of the two numbers in}
\eqref{numbers}.\vskip.1in

To see this, the basis \eqref{basiss2} for $S^2(S^p(\mathbb{R}^n))$
can be partitioned into classes
$$\{\mathbf{E}_\alpha\otimes_s\mathbf{E}_\beta:
\alpha+\beta=\gamma\}$$ for each $|\gamma|=2p$, with the number of
classes equal to $dim(S^{2p}(\mathbb{R}^n))$.  Beginning with a
distinguished member of a class, the same span is obtained by the
collection
\begin{equation}\mylabel{samespan}
\{\mathbf{E}_\alpha\otimes_s\mathbf{E}_\beta,\mathbf{E}_\alpha\otimes_s\mathbf{E}_\beta-\mathbf{E}_{\alpha'}\otimes_s\mathbf{E}_{\beta'},
\mathbf{E}_\alpha\otimes_s\mathbf{E}_\beta-\mathbf{E}_{\alpha''}\otimes_s\mathbf{E}_{\beta''},\ldots\}
\end{equation}
where $\alpha+\beta=\alpha'+\beta'=\cdots=\gamma$.  Every element
after the first is in the subspace $A^{2,p,n}$.

By the definition of $A^{2,p,n}$,\vskip.1in

\noindent\textit{Two representation matrices for the same
homogeneous polynomial of degree $2p$ always differ by a member of}
$A^{2,p,n}$.\vskip.1in

The members of the subspace $A^{2,p,n}$ when added to a
representation matrix for a polynomial \textit{change the
representation} of the polynomial and \textit{do not change the
polynomial}.  When a polynomial has an $sos$ representation, adding
what will be called a \textit{change} $\Delta$ to that
representation might or might not yield another $sos$
representation. In the case it does yield another, it cannot alter
the facts that the polynomials of degree $p$ that are squared share
or do not share a common \textit{real} root. That they share or do
not share a common complex root from $\mathbb{C}^n\setminus
\mathbb{C}\mathbb{R}^n$, however, possibly can be altered by adding
a $\Delta$.  Here $\mathbb{C}\mathbb{R}^n=\{a\mathbf{x}:
a\in\mathbb{C} \mbox{ and } \mathbf{x}\in\mathbb{R}^n\}$.

\begin{example}
$\Delta=\mathbf{E}_{(2,0)}\otimes_s\mathbf{E}_{(0,2)}-2\mathbf{E}_{(1,1)}^{\otimes
2}$ may be allowed to serve as the only basis element for
$A^{2,2,2}$. Letting $a\in\mathbb{R}$
\begin{equation}\mylabel{S_a}
\mathcal{S}_a:=(\mathbf{E}_{(2,0)}+\mathbf{E}_{(0,2)})^{\otimes
2}+a\Delta=\mathbf{E}_{(2,0)}+(1+a)\mathbf{E}_{(2,0)}\otimes_s\mathbf{E}_{(0,2)}+\mathbf{E}_{(0,2)}-2a\mathbf{E}_{(1,1)}^{\otimes
2}
\end{equation}
when applied to $\mathbf{x}^{\otimes 2}\mathbf{x}^{\otimes 2}$
always yield the $pd$ polynomial $(x_1^2+x_2^2)^2$.  Choosing a
linear order  $(2,0)\prec(0,2)\prec(1,1)$  for the basis elements of
$S^2(\mathbb{R}^2)$, the isomorphism \eqref{isomorphism}, of
$S^2(S^2(\mathbb{R}^2))$ with the symmetric $3\times 3$ matrices,
yields
$$\mathcal{S}_a=\left(\begin{array}{llc}1&1+a&0\\1+a&1&0\\0&0&-2a\end{array}\right)$$

The eigenvalues are $-a$, $2+a$ and $-2a$.  Using these together
with the corresponding unit eigenvectors suggests that \eqref{S_a}
be written $$
\mathcal{S}_a=-\dfrac{a}{2}(\mathbf{E}_{(2,0)}-\mathbf{E}_{(0,2)})^{\otimes
2}+\dfrac{2+a}{2}(\mathbf{E}_{(2,0)}+\mathbf{E}_{(0,2)})^{\otimes
2}-2a\mathbf{E}_{(1,1)}^{\otimes 2}$$ The representation matrix is
$psd$ if and only if $-2\leq a\leq0$ if and only if $$
\mathcal{S}_a\cdot\mathbf{x}^{\otimes 2}\mathbf{x}^{\otimes
2}=-\dfrac{a}{2}(x_1^2-x_2^2)^{ 2}+\dfrac{2+a}{2}(x^2_1+x_2^2)^{
2}-2a(x_1x_2)^{ 2}$$ is an $sos$ representaion.  Among these, each
quadratic term has the complex root $\mathbf{x}=(1,i)$ when $a=0$,
while there are no common complex roots when $-2\leq a < 0$.
\end{example}

This example  used the fact that a real symmetric $m\times m$ matrix
may be written as an element of $S^2(\mathbb{R}^m)$
\begin{equation}\mylabel{outerproduct}\sum^m_{j=1}\lambda_j \mathbf{u}^j\otimes \mathbf{u}^j\end{equation} where
the $\lambda_j$ are eigenvalues counted by multiplicity and
$\mathbf{u}^j\in \mathbb{R}^m$ are the corresponding \textit{unit }
eigenvectors.

The following proposition can be found in \cite{CLR95} p.106,
Proposition 2.3.  We include a proof in the multilinear language
used here.

\begin{prop}\mylabel{characterization}
A form $f\in\mathbb{R}[x_1,\ldots,x_n]$ of degree $2p$ is an $sos$
\emph{if and only if} there is a $psd$ representation matrix
$\mathcal{G}$ such that $f(x)=\mathcal{G}\cdot \mathbf{x}^{\otimes
p}\mathbf{x}^{\otimes p}$.
\end{prop}

\begin{proof}
When $\mathcal{G}$ is $psd$ and  a representation matrix for $f$,
then $\mathcal{G}$ can be written as a matrix
$\sum_{|\beta|=p}\lambda_\beta\mathbf{u}^\beta\otimes
\mathbf{u}^\beta$ where the $\mathbf{u}^\beta$ are the unit
eigenvectors with $\left(\begin{array}{c}n+p-1\\p\end{array}\right)$
real components $u^\beta_\alpha$ for $|\alpha|=p$ and $\lambda_\beta
\geq 0$ are the corresponding eigenvalues. By the isomorphism
\eqref{isomorphism} it is a  tensor
$\mathcal{G}=\sum_{|\beta|=p}\lambda_\beta(\sum_{|\alpha|=p}u^\beta_\alpha
\mathbf{E}_\alpha)^{\otimes 2}$ that acts as  $\mathcal{G}\cdot
\mathbf{x}^{\otimes p}\mathbf{x}^{\otimes
p}=\sum_{|\beta|=p}\lambda_\beta(\sum_{|\alpha|=p}u^\beta_\alpha
\mathbf{x}^\alpha)^{ 2}$.  Thus $f$ is $sos$.

If $f$ is $sos$, then it is a sum of forms
$$\left(\sum_{|\alpha|=p}a_\alpha \mathbf{x}^\alpha\right)^{
2}=\left(\sum_{|\alpha|=p}a_\alpha \mathbf{E}_\alpha\cdot
\mathbf{x}^{\otimes p}\right)^{ 2}=\left(\sum_{|\alpha|=p}a_\alpha
\mathbf{E}_\alpha\right)^{\otimes 2}\cdot \mathbf{x}^{\otimes
p}\mathbf{x}^{\otimes p}$$ $a_\alpha\in \mathbb{R}$.  $\mathcal{G}$
can be taken to be a sum of tensors $\left(\sum_{|\alpha|=p}a_\alpha
\mathbf{E}_\alpha\right)^{\otimes 2}$ each with a $psd$
representation matrix.
\end{proof}

\textit{A} psd \textit{representation matrix} $\mathcal{G}\in
S^2(S^p(\mathbb{R}^n))$ \textit{is also called a} Gram matrix.
\textit{For a form} $f$ \textit{of degree} $2p$ \textit{to be an}
$sos$ \textit{it is necessary and sufficient that it have a
representation} $f(x)=\mathcal{G}\cdot \mathbf{x}^{\otimes
p}\mathbf{x}^{\otimes p}$ \textit{for} some \textit{Gram matrix}
$\mathcal{G}$.\vskip.1in

An element of $S^2(S^p(\mathbb{R}^n))$ may also be viewed as a
linear transformation $\mathbf{t}\mapsto\mathcal{S}\mathbf{t}$ on
$S^p(\mathbb{R}^n)$ so that
$\mathcal{S}\cdot\mathbf{s}\mathbf{t}=\mathbf{s}\cdot\mathcal{S}\mathbf{t}$.

Two more elementary but useful observations follow from the
characterization of sums of squares given by Proposition
\ref{characterization} and elementary properties of $psd$ matrices.

\vskip.1in \textit{Suppose} $\mathcal{G}$ \textit{is a Gram matrix.
Then the form} $\mathcal{G}\cdot \mathbf{x}^{\otimes
p}\mathbf{x}^{\otimes p}$ \textit{is positive definite} if and only
if \textit{the tensor} $(\mathcal{G}+\Delta) \mathbf{x}^{\otimes
p}\neq \mathbf{0}$ \textit{for all nonzero}
$\mathbf{x}\in\mathbb{R}^n$ \textit{and for all changes}
$\Delta$.\vskip.1in

For $\mathbf{x}, \mathbf{y}\in \mathbb{R}^n$ put
$\mathbf{z}=\mathbf{x}+i\mathbf{y}\in \mathbb{C}^n$.  Then formally
using the binomial expansion
$$\mathbf{z}^{\otimes p}=\sum_{m=0}^p
\left(\begin{array}{c}p\\m\end{array}\right)i^m
Sym(\mathbf{x}^{\otimes (p-m)}\otimes \mathbf{y}^{\otimes
m})=$$$$\mathbf{x}^{\otimes
p}-\left(\begin{array}{c}p\\2\end{array}\right)Sym(\mathbf{x}^{\otimes
(p-2)}\otimes \mathbf{y}\otimes \mathbf{y})+\cdots+i(p\;
Sym(\mathbf{x}^{\otimes (p-1)}\otimes \mathbf{y})-\cdots)$$$$:=Re\;
\mathbf{z}^{\otimes p}+i Im \;\mathbf{z}^{\otimes p}$$  A linear
transformation on $S^p(\mathbb{ R}^n)$ is extended to complex valued
tensors by $\mathcal{S}(\mathbf{s}+i\mathbf{t})=
\mathcal{S}\mathbf{s}+i\mathcal{S}\mathbf{t}$. It follows that
$\Delta\cdot \mathbf{z}^{\otimes p}\mathbf{z}^{\otimes p}=0$ for all
changes $\Delta$.  This is because the coefficients on the powers of
the  \textit{real} variable $t$ in $\Delta\cdot
(\mathbf{x+\emph{t}y})^{\otimes p}(\mathbf{x+\emph{t}y})^{\otimes
p}=0$ must all vanish. The same coefficients occur on the unreduced
powers of $i$ in $\Delta\cdot \mathbf{z}^{\otimes
p}\mathbf{z}^{\otimes p}$.  Or one can invoke the multi-index
formalism.  Similarly, by comparing coefficients between binomial
expansions, \eqref{actsmulti} extends to complex rank-one tensors
$$\mathbf{E}_\alpha\cdot \mathbf{z}^{\otimes p}=\mathbf{z}^\alpha$$

\begin{multline}\mylabel{strategy}
\mbox{\textit{Let} $\mathcal{S}$ \textit{be a representation matrix.
Then} $\mathcal{S}\cdot \mathbf{x}^{\otimes p}\mathbf{x}^{\otimes
p}$ \textit{is a \emph{coercive}} $sos$ if and only if}
\\\mbox{\textit{there} \textit{exists a} $\Delta$ \textit{such that}
$\mathcal{S}+\Delta$ \textit{is a Gram matrix, and for every
nonzero} $z\in \mathbb{C}^n$ \textit{the}\hskip.2in
}\\\mbox{\textit{tensor} $(\mathcal{S}+\Delta) \mathbf{z}^{\otimes
p}\neq \mathbf{0}$.\hskip3.75in}
\end{multline}
For when $\mathcal{S}+\Delta$ is a Gram matrix it may be written
$\sum \mathbf{g}_j\otimes\mathbf{g}_j$ with the collection of
$\mathbf{g}_j\in S^p(\mathbb{R}^n)$ linearly independent; and
$(\mathcal{S}+\Delta) \mathbf{z}^{\otimes p}=\sum
(\mathbf{g}_j\cdot\mathbf{z}^{\otimes p})\mathbf{g}_j$.\vskip.1in

The strategy, then, for showing that a positive definite $sos$ is a
coercive $sos$ is to change the Gram matrix, preserving its $psd$
property, in order to eliminate from the null space all
2-dimensional subspaces of the form $span\{\mathbf{s},\mathbf{t}\}$
where $\mathbf{s}+i\mathbf{t}=\mathbf{z}^{\otimes p}$ for nonzero
$\mathbf{z}\in \mathbb{C}$.  In this way the point of view of this
article is opposite that of some literature growing out of Hilbert's
theorems on sums of squares.  For example, the coercive result
\eqref{positiveresult} is achieved by eliminating the nontrivial
null space altogether, i.e. showing that $pd$ Gram matrices exist
for those cases.  On the other hand, the most remarkable and
difficult result of Hilbert's is that for the cone $P_{3,4}$, where
the rank of a Gram matrix can be as large as $6$, every polynomial
can be written a  sum of just $3$ squares.  Out of this came the
general idea of the \textit{length} or minimum number of squares
required for an $sos$ representation and out of this the
\textit{Pythagoras number}, the minimum number of squares needed
over a collection of $sos$ polynomials.  See, for example,
\cite{BCR98}, \cite{CLR95}, \cite{Pfi95}, \cite{PD01} and others.

For coerciveness the length of an $sos$ is often an undesirable
number, and one naturally wishes to \textit{maximize} the number of
independent squares in a representation.  That this is an
interesting problem is shown here by demonstrating,  in the case of
a \textit{positive definite} polynomial with $psd$ representation
(Gram) matrix, that the rank of its Gram matrices cannot  in general
be increased enough to achieve the desired end, vis. coerciveness.

We end this section by restating question \eqref{question} in
multilinear language and by outlining the construction by which the
answer is shown to be \textit{no} in general.\vskip.1in

\textit{Suppose} $\mathcal{G}\in S^2(S^p(\mathbb{R}^n))$ \textit{is
a Gram matrix}
 \textit{and} $\mathcal{G}\mathbf{x}^{\otimes p}\neq \mathbf{0}$
\textit{for all rank-one tensors} .  \textit{Does there exist a
change} $\Delta$ \textit{such that} $\mathcal{G}+\Delta$ \textit{is}
 \textit{a Gram matrix and} $(\mathcal{G}+\Delta)\mathbf{z}^{\otimes p}\neq
\mathbf{0}$ \textit{for all nonzero} $z\in \mathbb{C}$?\vskip.1in

Or less precisely, can a Gram matrix $\mathcal{G}$ that is $pd$ on
the rank-one tensors be \textit{changed} to be a Gram matrix that is
$pd$ on all subspaces of the form $span\{\mathbf{s},\mathbf{t}\}$
where $\mathbf{s}+i\mathbf{t}=\mathbf{z}^{\otimes p}$ for some
nonzero $\mathbf{z}\in \mathbb{C}$?

The question is answered below in the negative, for the cases
$n\geq6$, $p= 2$ and $n\geq4$, $p= 3$, by the construction
\begin{multline}\mylabel{construction}
\mbox{Construct a Gram matrix $\mathcal{G}$ such that}\\\mbox{(i)
$\mathcal{G}$ is positive definite on the rank-one tensors.
}\\\mbox{\;\;\;\;\;\;\;\;\;\;\;\;\;\;\;\;\;\;\;\;\;\;\;\;\;\;\;\;\;\;\;\;\;\;\;(ii)
there exists a nonzero $\mathbf{z}\in \mathbb{C}^n$ such that the
tensor $\mathcal{G}\mathbf{z}^{\otimes p}=\mathbf{0}$.}\\\mbox{(iii)
$\mathcal{G}+\Delta$ is never a Gram matrix whenever
$\Delta\mathbf{z}^{\otimes p}\neq \mathbf{0}$.}
\end{multline}

A uniqueness condition stronger than (iii) is $$ \mbox{(iii)$'$
$\mathcal{G}+\Delta$ is never a Gram matrix whenever $\Delta\neq
\mathbf{0}$.}$$

\section{A positive definite quartic with a unique Gram
matrix}\mylabel{quartic}

In this section an element of $\Sigma_{4,4}$ is constructed that
satifies (i) and (iii)$'$ of the construction \eqref{construction},
but not (ii).

The vector space of representation matrices $S^2(S^p(\mathbb{R}^n))$
inherits a topology from the Euclidean space of the same dimension.
The closed cone of Gram matrices will have as its interior the cone
of positive definite Gram matrices.  The boundary of this cone is
the set of Gram matrices with rank less than
$\left(\begin{array}{c}n+p-1\\p\end{array}\right)$.

Part (ii) of the construction \eqref{construction} cannot be
realized if $\mathcal{G}$ is taken in the interior of the cone. Thus
$\mathcal{G}$ must be on the boundary if one hopes to realize (ii)
and one is led to consider $pd$ polynomials of degree $2p$ that
border those that are not sums of squares.  Historically $pd$ and
$psd$ polynomials that are not $sos$ are difficult to locate.  It is
therefore sensible to begin with a known $pd$ polynomial that is not
$sos$, i.e. does not have a Gram matrix but is definite on the
rank-one tensors, and perturb it in such a way so that one arrives
at the boundary of the Gram matrices while maintaining the rank-one
definiteness. Here we take $n=4, p=2$, let $\mathbf{x}\in
\mathbb{R}^4$ correspond to $(w,x,y,z)$ and begin with the Choi-Lam
quartics $q_\eta$ \eqref{q}, \eqref{qeta}, letting $\eta$ increase
until the quartic \eqref{sosquartic} is achieved.

Except for the uniqueness of representation claim, all other claims
made for \eqref{sosquartic} in Section \ref{Introduction} can be
quickly proved.\vskip.1in

\noindent 1. By expanding the right side of \eqref{sosquartic} and
collecting terms the right side meets the definition of $q_{\eta_0}$
\eqref{qeta} if the coefficients on the $x^2y^2, y^2z^2$ and $
z^2x^2$ terms equal $1$.  This occurs when

\noindent 2. $\sqrt{\eta_o}$ is a root of
$X^3-\frac{1}{2}X+\frac{1}{9}=0$.

\noindent 3. $\sqrt{\eta_o}$ must be chosen to be the smallest
positive root, else $\eta_0$ would not be the smallest $\eta$ that
makes $q_\eta$ an $sos$.\vskip.1in

Since  degree and dimension are low in this section, tensors
$\mathbf{E}_\alpha$ will be denoted by using only the entries of
each multi-index as subscripts, as in $\mathbf{E}_{ijkl}$ instead of
$\mathbf{E}_{(i,j,k,l)}$.  Thus
$\mathbf{E}_{2000}\cdot\mathbf{x}^{\otimes 2}=x_1^2=w^2$,
etc.\vskip.1in

\noindent 4. That $\eta_0$, as described in Claims 2 and 3, is the
\textit{smallest} $\eta$ for which $q_\eta$ is an $sos$ will follow
once it is proved that
\begin{multline}\mylabel{Qeta0}
\mathcal{Q}_{\eta_0}=\left(\mathbf{E}_{2000}-\sqrt{\eta_0}(\mathbf{E}_{0200}+\mathbf{E}_{0020}+\mathbf{E}_{0002})\right)^{\otimes
2}+\\\dfrac{2}{9\sqrt{\eta_0}}\left[(3\sqrt{\eta_0}\mathbf{E}_{1100}-\mathbf{E}_{0011})^{\otimes
2}+(3\sqrt{\eta_0}\mathbf{E}_{1010}-\mathbf{E}_{0101})^{\otimes
2}+(3\sqrt{\eta_0}\mathbf{E}_{1001}-\mathbf{E}_{0110})^{\otimes
2}\right]
\end{multline}
is the \textit{unique} Gram matrix $\mathcal{G}$ for which
$q_{\eta_0}(\mathbf{x})=\mathcal{G}\cdot\mathbf{x}^{\otimes
2}\mathbf{x}^{\otimes 2}$.  For if $q_\eta$ were an $sos$ for some
$\eta<\eta_0$, then
\begin{equation}\mylabel{2different}q_{\eta_0}=q_\eta+(\eta_0-\eta)(x^4+y^4+z^4)=q_\eta+(\eta_0-\eta)((x^2-y^2)^2+(\sqrt{2}xy)^2+z^4)\end{equation}
and the polynomial identity presents two different Gram matrices for
$q_{\eta_0}$.  Letting $\mathcal{Q}_{\eta}$ be, by Proposition
\ref{characterization}, a Gram matrix for $q_\eta$, $q_{\eta_0}$ now
has both
\begin{equation}\mylabel{identical1}
\mathcal{Q}_\eta+(\eta_0-\eta)(\mathbf{E}_{0200}^{\otimes
2}+\mathbf{E}_{0020}^{\otimes 2}+\mathbf{E}_{0002}^{\otimes 2})
\end{equation}
and
\begin{equation}\mylabel{identical2}
\mathcal{Q}_\eta+(\eta_0-\eta)((\mathbf{E}_{0200}-\mathbf{E}_{0020})^{\otimes
2}+2\mathbf{E}_{0110}^{\otimes 2}+\mathbf{E}_{0002}^{\otimes 2})
\end{equation}
as Gram matrices.  They differ by
$\Delta=(\eta_0-\eta)(2\mathbf{E}_{0110}^{\otimes
2}-\mathbf{E}_{0200}\otimes_s\mathbf{E}_{0020})$ contradicting the
uniqueness of $\mathcal{Q}_{\eta_0}$.

\begin{rem}
In contrast, the identity $2x^4+2y^4=(x^2-y^2)^2+(x^2+y^2)^2$
suggests $2\mathbf{E}_{0200}^{\otimes 2}+2\mathbf{E}_{0020}^{\otimes
2}$ and $(\mathbf{E}_{0200}-\mathbf{E}_{0020})^{\otimes
2}+(\mathbf{E}_{0200}+\mathbf{E}_{0020})^{\otimes 2}$ which are
identical Gram matrices.  The two polynomial expressions are said to
be obtained from one another by \textit{orthogonal transformation}.
See Proposition 2.10 of \cite{CLR95}, p.108.  It is for this reason
that by themselves it is not clear that each of \eqref{identical1}
or \eqref{identical2} differs from $\mathcal{Q}_{\eta_0}$ since
$\mathcal{Q}_{\eta}$ is unspecified.
\end{rem}

\noindent 5. That $q_{\eta_0}$ is \textit{coercive} is seen by
showing that the corresponding homogeneous system of four quadratic
equations has no solution in $\mathbb{C}^4\setminus\{\mathbf{0}\}$.
One starts with assuming a solution $(w,x,y,z)$ has one of its
coordinates equal to zero, cases that can be quickly eliminated.
Then, assuming a solution has all nonzero coordinates, one has by
using the last three quadratics of \eqref{sosquartic},
$y^2z=3\sqrt{\eta_0}wxy=zx^2$ etc., whence $x^2=y^2=z^2$, whence
$3\sqrt{\eta_0}|w|=|x|$ by any of the last three quadratics.  Then
$|w|^2=3\sqrt{\eta_0}|x|^2$ by the first, whence
$\sqrt{\eta_0}=\frac{1}{3}$ which is not true by Claim 2.\vskip.1in

The only task remaining is to prove the uniqueness of the Gram
matrix $\mathcal{Q}_{\eta_0}$.  Before that is done a bit more will
be said about finding \eqref{sosquartic}.

An initial choice of representation matrices for the forms $q_\eta$
is
\begin{multline}\mylabel{initial}
 \mathcal{S}_\eta=\mathbf{E}_{2000}^{\otimes
2}+\mathbf{E}_{0110}^{\otimes 2}+\mathbf{E}_{0011}^{\otimes
2}+\mathbf{E}_{0101}^{\otimes
2}\\-\dfrac{2}{3}(\mathbf{E}_{1100}\otimes_s
\mathbf{E}_{0011}+\mathbf{E}_{1010}\otimes_s
\mathbf{E}_{0101}+\mathbf{E}_{1001}\otimes_s
\mathbf{E}_{0110})+\eta(\mathbf{E}_{0200}^{\otimes
2}+\mathbf{E}_{0020}^{\otimes 2}+\mathbf{E}_{0002}^{\otimes 2})
\end{multline}

The $q_\eta$ are symmetric in $x,y$ and $z$.  As $\eta$ increases,
if $\mathcal{G}$ becomes the first Gram matrix encountered so would
be $\mathcal{G'}$ where $\mathcal{G'}$ is derived from $\mathcal{G}$
by permuting the indices for $x,y$ and $z$.  Averaging all such
permutations would produce a first Gram matrix that was symmetric in
$x,y$ and $z$.  Therefore the symmetry in the choice of
$\mathcal{S}_\eta$ is no loss of generality, and we expect that if a
Gram matrix uniquely represents a $q_\eta$, then  it will  be
symmetric in $x,y$ and $z$.

Arrange the basis elements $\mathbf{E}_{2000},\ldots$ according to
the linear order $w^2\prec x^2\prec y^2\prec z^2\prec wx\prec
yz\prec wy\prec zx\prec wz\prec xy$.  Then the matrix for
$\mathcal{S}_\eta$ with respect to the basis \eqref{E} is
\begin{equation}\mylabel{10by10}\left(\begin{array}{crrrclclcl}1&-b&-b&-b&&&&&&\\-b&\eta
&a&a&&&&&&\\-b&a&\eta &a&&&&&&\\-b&a&a&\eta
&&&&&&\\&&&&2b&-\frac{2}{3}&&&&
\\\\&& &&-\frac{2}{3}&1-2a&&&&\\&&&&&&2b&-\frac{2}{3}&&
\\\\&&&& &&-\frac{2}{3}&1-2a&&\\&&&&&&&&2b&-\frac{2}{3}
\\\\&&&&&& &&-\frac{2}{3}&1-2a\end{array}\right)\end{equation}
when the\textit{ parameters} $a=b=0$.  The unmarked entries are
$zero$.

The two parameters permit the addition of six \textit{changes} in a
way that also obey the symmetry considerations in $x,y$ and $z$. The
smallest value of $\eta$ that allows a choice of $a$ and $b$ so that
each of the four block matrices becomes rank-$1$ and $psd$ is the
$\eta_0$ defined above.  The minimizing choices are $a=\eta_0$ and
$b=\sqrt{\eta_0}$.

There are, however, \textit{twenty} independent changes $\Delta$ in
$S^2(S^2(\mathbb{R}^4))$ altogether.  Though the type of argument
being given can be made rigorous and lead to a uniqueness proof for
$\mathcal{Q}_{\eta_0}$, we will instead present another argument
which will also be elementary, but also clearly decisive while
computationally not too long if  Maple\tiny$^{TM}$\normalsize$ 10$
is used.  It is based on the observation \vskip.1in

\textit{Suppose} $\mathcal{G}$ \textit{is a Gram matrix.  Then a
necessary (but not sufficient) condition for} $\mathcal{G}+\Delta$
\textit{to be a Gram matrix is that} $\Delta$ \textit{be} psd
\textit{on} $Null(\mathcal{G})$, \textit{the null space of}
$\mathcal{G}:S^p(\mathbb{R}^n)\rightarrow S^p(\mathbb{R}^n)$,
\textit{i.e. for every} $\mathbf{t}\in Null(\mathcal{G})$ \textit{it
is necessary that} $\Delta \cdot \mathbf{t}\mathbf{t}\geq
0$.\vskip.1in

Let $N$ be a nonempty subspace of $S^p(\mathbb{R}^n)$.  When
$\mathcal{S}\cdot \mathbf{t}\mathbf{t}\geq 0$ fails to hold for some
$\mathbf{t}\in N$ while $\mathcal{S}\cdot \mathbf{s}\mathbf{s}> 0$
for an $\mathbf{s}\in N$, $\mathcal{S}$ is said to be \textit{not
definite} on $N$.  Thus \begin{multline}\mylabel{uniqueness?}\mbox{
\textit{If} $f(x)=\mathcal{G}\cdot \mathbf{x}^{\otimes
p}\mathbf{x}^{\otimes p}$ \textit{where} $\mathcal{G}$ \textit{is a
Gram matrix and if} every \textit{nonzero} $\Delta\in A^{2,p,n}$
\textit{is} }\\\mbox{ not definite \textit{on}
$Null(\mathcal{G})$,\textit{ then} $\mathcal{G}$ \textit{is the}
unique \textit{Gram matrix for} $f$.\hskip2in}\end{multline}

This is in fact  a statement about subspaces of $S^p(\mathbb{R}^n)$
and the Gram matrices that can be supported on their orthogonal
complements.  Consequently\vskip.1in

\textit{Let} $N$ \textit{be a subspace of} $S^p(\mathbb{R}^n)$
\textit{and} $\{\mathbf{t}_1,\ldots,\mathbf{t}_r\}$ \textit{a basis
for its orthogonal complement} $M$.  \textit{Suppose} every
\textit{nonzero} $\Delta\in A^{2,p,n}$ \textit{is} not definite
\textit{on} $N$.  \textit{Let} $T$ \textit{be any linear
transformation on} $M$. \textit{Then}
$\mathcal{G}_T=(T(\mathbf{t}_1))^{\otimes 2}+\cdots
+(T(\mathbf{t}_r))^{\otimes 2}$  \textit{is the} unique \textit{Gram
matrix for the} $sos$ $f_T(\mathbf{x})=\mathcal{G}_T\cdot
\mathbf{x}^{\otimes p}\mathbf{x}^{\otimes p}$.  \textit{The
collection of all such} $f_T$ \textit{is a convex cone of}
$\Sigma_{n,2p}$.

The last statement follows because if $\mathcal{G}_T$ and
$\mathcal{G}_U$ are $psd$ on $M$ so is their sum which will be given
by some $\mathcal{G}_V$ with the linear transformation $V$ on $M$
derived, for example, by using \eqref{outerproduct}.

\begin{rem} If, for example, $I$ is the identity on $M$ and $U$ is
an orthogonal transformation on $M$, then $f_I=f_U$.  This is
Proposition 2.10 of \cite{CLR95} again.
\end{rem}

Given a subspace $N\subset S^p(\mathbb{R}^n)$ of dimension $m$ the
following steps will be carried out in order to prove that certain
\textit{sums of squares},  supported like the above $f_T$ on the
orthogonal complement of $N$, have unique Gram matrices. \vskip.1in
\noindent 1. Form a general linear combination
$\mathbf{t}=a\mathbf{t}_1+b\mathbf{t}_2+\cdots$  of the $m$ basis
elements of $N$.\vskip.1in \noindent 2. Apply each element $\Delta$
of a basis  for  $A^{2,p,n}$ \eqref{deltas} to the general linear
combination, as $\Delta\cdot \mathbf{t}\mathbf{t}$, yielding a set
of homogeneous quadratic polynomials in the $m$ \textit{variables}
$a,b,\ldots$\vskip.1in \noindent 3. Thinking of each quadratic
polynomial from Step 2 as a \textit{linear} expression in the
monomials $a^2, b^2,\ldots, ab, ac,\ldots, bc,bd, \ldots$, write the
$\left(\begin{array}{c}\left(\begin{array}{c}n+p-1\\p\end{array}\right)+1\\\\2\end{array}\right)-\left(\begin{array}{c}n+2p-1\\2p\end{array}\right)$
by $\left(\begin{array}{c}m+1\\2\end{array}\right)$ coefficient
matrix for these linear expressions.\vskip.1in \noindent 4. Bring
the coefficient matrix of Step 3 to reduced row echelon form thereby
obtaining a set of quadratic polynomials that is equivalent to the
set of Step 2, i.e. each set of quadratics consists of only linear
combinations of quadratics from the other.\vskip.1in \noindent 5.
Show that no nontrivial linear combination of the quadratics from
Step 4 yields a definite or semi-definite quadratic in the $m$
variables.

\begin{rem}
Steps 1 through 4 can be thought of as supplying details for an
algorithm designed to show a certain semi-algebraic set consists
(here) of one point (the origin).  See the second algorithmic step
and the remark that follows on p. 101 of \cite{PW98}.  Here it is
Step 5 that is uncertain.
\end{rem}

In the case of interest here, there are $m=6$ variables
$a,b,c,d,e,f$ and the coefficient matrix is $20\times 21$, more
quadratic monomials than quadratic polynomials.

To simplify calculation, $\mathbb{R}^4$ (and thus \eqref{Qeta0}) is
scaled in the variable $w$, replaced with
$\frac{w}{3\sqrt{\eta_0}}$. Define
$$\gamma_0:=27\eta_0^{3/2}$$
Then \eqref{Qeta0} is a linear combination with\textit{ positive}
coefficients of the tensors
\begin{multline}\mylabel{gammatensors}\left(3\mathbf{E}_{2000}-\gamma(\mathbf{E}_{0200}+\mathbf{E}_{0020}+\mathbf{E}_{0002})\right)^{\otimes
2},\\(\mathbf{E}_{1100}-\mathbf{E}_{0011})^{\otimes
2},(\mathbf{E}_{1010}-\mathbf{E}_{0101})^{\otimes 2},\mbox{ and
}(\mathbf{E}_{1001}-\mathbf{E}_{0110})^{\otimes 2}\end{multline}
when $\gamma=\gamma_0$.  By Claims 2 and 3 at the beginning of this
section the estimate $\sqrt{\eta_0}<1/3$ holds, whence
 $0<\gamma_0 <1$.  Thus all assertions about $q_{\eta_0}$
 \eqref{sosquartic} will hold once the following theorem is proved.
\begin{thm}\mylabel{gammaquartic}
Given any $\gamma$, $0<\gamma <1$, and any choice of $a_j>0,\;
j=1,2,3,4$, the quartic form of  \;$\mathbb{R}[w,x,y,z]$
\begin{multline}\mylabel{conerep2}
a_1(3w^2-\gamma(x^2+y^2+z^2))^2+a_2(wx-yz)^2+a_3(wy-zx)^2+a_4(wz-xy)^2
\end{multline}
is coercive and has a unique Gram matrix.
\end{thm}

\begin{proof}
Coerciveness follows as for $q_{\eta_0}$ in Claim 5 at the beginning
of this section.

Fix any $0<\gamma <1$ and denote by $\mathcal{G}_\gamma$ any linear
combination, with \textit{positive} coefficients, of the tensors
\eqref{gammatensors}.  A basis for the null space of
$\mathcal{G}_\gamma$ is supplied by
\begin{multline*}\mathbf{E}_{1100}+\mathbf{E}_{0011},\;\mathbf{E}_{1010}+\mathbf{E}_{0101},\;\mathbf{E}_{1001}+\mathbf{E}_{0110},\;
\gamma\mathbf{E}_{2000}+\mathbf{E}_{0200}+\mathbf{E}_{0020}+\mathbf{E}_{0002},\\\mathbf{E}_{0200}-\mathbf{E}_{0020},\mbox{
and }\mathbf{E}_{0200}-\mathbf{E}_{0002}\end{multline*} as
\eqref{dimsp}, \eqref{normalbasis} and \eqref{ndotn} show.  A
general linear combination of these is
$\mathbf{g}=2a\mathbf{E}_{1100}+2a\mathbf{E}_{0011}+2b\mathbf{E}_{1010}+2b\mathbf{E}_{0101}+2c\mathbf{E}_{1001}+2c\mathbf{E}_{0110}+
\gamma
d\mathbf{E}_{2000}+(d+e+f)\mathbf{E}_{0200}+(d-e)\mathbf{E}_{0020}+(d-f)\mathbf{E}_{0002}$

A basis for the changes $A^{2,2,4}$ divides into three sets
depending on the number of multi-indices $\alpha$ with $\alpha !=2$
that are used to express a $\Delta$.  The first type has two such
$\alpha$ as in $$\mathbf{E}_{0110}^{\otimes
2}-\frac{1}{2}\mathbf{E}_{0200}\otimes_s\mathbf{E}_{0020}$$ there
are $6$ of these altogether. The second type uses one as in
$$\frac{1}{2}\mathbf{E}_{2000}\otimes_s\mathbf{E}_{0110}-\frac{1}{2}\mathbf{E}_{1100}\otimes_s\mathbf{E}_{1010}$$
There are $12$ of these.  Finally there are only $2$ independent
changes that use no $\alpha !=2$.  We will use
$$\frac{1}{2}\mathbf{E}_{1100}\otimes_s\mathbf{E}_{0011}-\frac{1}{2}\mathbf{E}_{1010}\otimes_s\mathbf{E}_{0101}\mbox{
and
}\frac{1}{2}\mathbf{E}_{1100}\otimes_s\mathbf{E}_{0011}-\frac{1}{2}\mathbf{E}_{1001}\otimes_s\mathbf{E}_{0110}$$
The last type was used implicitly in the initial choice
\eqref{initial}.  The first type was introduced by the parameters
in \eqref{10by10}.

Keeping in mind that by \eqref{normalbasis} and \eqref{ndotn}
$\mathbf{E}_\alpha\cdot\mathbf{E}_\alpha=\frac{\alpha !}{2}$ and
computing $\Delta\cdot \mathbf{g}\mathbf{g}$ we obtain \vskip.1in

$a^2-\gamma d(d+e+f)$

$b^2-\gamma d(d-e)$

$c^2-\gamma d(d-f)$

$a^2-(d-e)(d-f)$

$b^2-(d+e+f)(d-f)$

$c^2-(d+e+f)(d-e)$ \vskip.1in then\vskip.1in

$\gamma da-bc$

$\gamma db-ac$

$\gamma dc-ab$

$(d+e+f)a-bc$

$(d+e+f)b-ac$

$(d+e+f)c-ab$

$(d-e)a-bc$

$(d-e)b-ac$

$(d-e)c-ab$

$(d-f)a-bc$

$(d-f)b-ac$

$(d-f)c-ab$ \vskip.1in and then\vskip.1in

$a^2-b^2$

$a^2-c^2$\vskip.1in

Linearly ordering the monomial squares in alphabetical order
followed by the indefinite monomials in alphabetical order
$a^2,b^2,\ldots,f^2,ab,ac,\ldots,af,bc,\ldots,df,ef$ the $20\times
21$ coefficient matrix of Step 3 above is obtained.  Passing to
reduced row echelon form, a matrix that consists of a $20\times 20$
\textit{identity} matrix together with a \textit{21st column} with
successive entries $$\dfrac{\gamma}{1-\gamma}\mbox{,
}\dfrac{\gamma}{1-\gamma}\mbox{, }\dfrac{\gamma}{1-\gamma}\mbox{,
}\dfrac{1}{1-\gamma}\mbox{, }2\mbox{,
}2,0,0,0,0,0,0,0,0,0,0,0,0,0,0$$ is obtained.

Thus an equivalent set of quadratic polynomials is

\begin{multline}\mylabel{equivpolys}
$$\hskip1.7in a^2+\dfrac{\gamma}{1-\gamma}ef$$\\$$b^2+\dfrac{\gamma}{1-\gamma}ef$$\\
$$c^2+\dfrac{\gamma}{1-\gamma}ef$$\\$$\hskip2in d^2+\dfrac{1}{1-\gamma}ef\hskip2in$$\\$$e^2+2ef$$\\$$f^2+2ef\hskip2.1in$$
\end{multline}
together with the collection of $14$ indefinite monomials
$ab,ac,\ldots,df$ ($ef$ not included). Precisely when $0<\gamma<1$
is there no nontrivial linear combination of these that yields a
definite or semi-definite quadratic polynomial.  Thus uniqueness
follows from \eqref{uniqueness?}.
\end{proof}

More generally, the quartics \eqref{conerep2} are $pd$ whenever
$\gamma\neq 0$ and $\gamma\neq 1$.  When $\gamma<0$, expanding the
first square makes it transparent that the quartics \eqref{conerep2}
have \textit{positive definite} Gram matrices and are thus coercive
$sos$.  When $\gamma >1$ it is not clear in this way, but it is
clear from \eqref{equivpolys} that there is a $\Delta$ that is
positive definite on the null space of the $\mathcal{G}_\gamma$
(from the proof) that represents a \eqref{conerep2}.  By taking
$\epsilon>0$ small enough $\mathcal{G}_\gamma+\epsilon\Delta$ will
be $pd$ by the proposition below.

In some cases there only exist nontrivial $\Delta$ that are positive
\textit{semi-definite} on the null space of a $psd$ $\mathcal{G}$.
In those cases the proposition below gives necessary and sufficient
conditions for $\mathcal{G}+\epsilon\Delta$ to be $psd$, i.e. for
the associated $sos$ to \textit{not} have a unique Gram matrix.
When $Null(\mathcal{G})\cap Null(\Delta)\neq Null(\mathcal{G})$ the
propsition gives necessary and sufficient conditions for
$\mathcal{G}+\epsilon\Delta$ to be $psd$ with greater rank than
$\mathcal{G}$.  It provides conditions to \textit{build up} the
ranks of Gram matrices associated to an $sos$ in an attempt to prove
coerciveness of the $sos$.

The \textit{length} of a vector $\mathbf{x}\in \mathbb{R}^m$ is
denoted $|\mathbf{x}|$ and the \textit{operator norm} of an $m\times
m$ matrix $B$, as a transformation on $\mathbb{R}^m$, is denoted
$|B|=\max_{|\mathbf{x}|=1}|B\mathbf{x}|$.

\begin{prop}\mylabel{matrixprop}
Let $A$ be real symmetric positive semi-definite $m\times m$ matrix.
Let $B$ be real symmetric  $m\times m$ matrix that is $psd$ on
$Null(A)\subset \mathbb{R}^m$, i.e. $\mathbf{z}\cdot
B\mathbf{z}\geq0$ for all $\mathbf{z}\in Null(A)$.

Then for all $\epsilon>0$ small enough $A+\epsilon B$ is a positive
semi-definite matrix \emph{if and only if} whenever $\mathbf{z}_1\in
Null(A)$ and $\mathbf{z}_1\cdot B\mathbf{z}_1=0$ it follows that
$B\mathbf{z}_1=\mathbf{0}$.

In the case $A+\epsilon B$ is $psd$ $Null( A+\epsilon B)\subset
Null(A)$ for all $\epsilon >0$ small enough, with strict containment
when $\mathbf{z}\cdot B\mathbf{z}$ does not vanish for every
$\mathbf{z}\in Null(A)$.

If $B$ is $pd$ on $Null(A)$ then $A+\epsilon B$ is $pd$ for all
$\epsilon >0$ small enough.
\end{prop}

\begin{proof}
$A$ and $B$ are assumed nontrivial.  The last statement is proved
first.

Let $a>0$ be the smallest nonzero eigenvalue of $A$.  Let $b>0$ be
the smallest number satisfying $\mathbf{z}\cdot B\mathbf{z} \geq
b|\mathbf{z}|^2$ for all $\mathbf{z}\in Null(A)$.  Each
$\mathbf{x}\in \mathbb{R}^m$ has a unique decomposition
$\mathbf{x}=\mathbf{y}+\mathbf{z}$ where $\mathbf{z}\in Null(A)$ and
$\mathbf{y}$ is orthogonal to $Null(A)$, i.e. by the symmetry of
$A$, each
 $\mathbf{y}$ is a sum of the eigenvectors of $A$ that have \textit{positive}
 eigenvalues.  Thus
 \begin{multline}\mylabel{below}
\mathbf{x}\cdot(A+\epsilon B)\mathbf{x}=\mathbf{y}\cdot A
\mathbf{y}+\epsilon\mathbf{y}\cdot B
\mathbf{y}+2\epsilon\mathbf{y}\cdot B
\mathbf{z}+\epsilon\mathbf{z}\cdot B
\mathbf{z}\geq\\a|\mathbf{y}|^2-\epsilon |B||\mathbf{y}|^2-2\epsilon
|B| |\mathbf{y}||\mathbf{z}|+\epsilon b|\mathbf{z}|^2
 \end{multline}
 For $\mathbf{x}\neq\mathbf{0}$ this last quantity will always be
 positive for any $\epsilon$ satisfying $0< \epsilon <
 \frac{ab}{|B|^2+b|B|}$, proving the positive definiteness of $A+\epsilon
 B$.

 Now assume $B$ is $psd$ on $Null(A)$.  The first conclusion is
 proved next.

 Assume for some $\epsilon >0$ that $A+\epsilon
 B$ is $psd$.  Let $\mathbf{z}_0\in Null(A)$ and assume $\mathbf{z}_0\cdot B
\mathbf{z}_0=0$.  Thus $\mathbf{z}_0\cdot(A+\epsilon
B)\mathbf{z}_0=0$.  Since $A+\epsilon
 B$ has a $psd$ square root it follows that
$(A+\epsilon B)\mathbf{z}_0=0$ whence $B\mathbf{z}_0=\mathbf{0}$.

For the other direction and for each $\mathbf{x}\in \mathbb{R}^m$,
with $\mathbf{x}=\mathbf{y}+\mathbf{z}$ as before, the
\textit{equality} in \eqref{below} is again obtained.  Each
$\mathbf{z}\in Null(A)$ has a unique decomposition
$\mathbf{z}=\mathbf{z}_0+\mathbf{z}_1$ where $\mathbf{z}_0\in
Null(A)\cap Null(B)$ and $\mathbf{z}_1\in Null(A)$ is
\textit{orthogonal  } to $Null(A)\cap Null(B)$.  In the event
$Null(A)\cap Null(B)=Null(A)$ it follows that
$\mathbf{z}=\mathbf{z}_0$ and \eqref{below} yields
$\mathbf{x}\cdot(A+\epsilon B)\mathbf{x}\geq
a|\mathbf{y}|^2-\epsilon |B||\mathbf{y}|^2\geq 0$ for every
$\mathbf{x}$ if $\epsilon$ is small enough, with vanishing occurring
only when $\mathbf{x}\in Null(A)$.  Otherwise there is a smallest
number $b_1>0$ such that $\mathbf{z}_1\cdot B\mathbf{z}_1 \geq
b_1|\mathbf{z}_1|^2$ for all $\mathbf{z}_1\in Null(A)$ orthogonal to
$Null(A)\cap Null(B)$.  This follows by the hypothesis,
$\mathbf{z}_1\cdot B\mathbf{z}_1=0$ implies
$B\mathbf{z}_1=\mathbf{0}$, whence $\mathbf{z}_1\in Null(A)\cap
Null(B)$ whence $\mathbf{z}_1=\mathbf{0}$.  Consequently
$\mathbf{z}$ may be replaced by $\mathbf{z}_1$ and $b$ by $b_1$ in
\eqref{below}.  For all $\mathbf{x}\notin Null(A)\cap Null(B)$ and
$\epsilon >0$ small enough \eqref{below} is then \textit{positive},
completing the proof of the first conclusion.

It has been shown for $\epsilon >0$ small enough that positivity of
\eqref{below} fails only when $\mathbf{x}\in Null(A)\cap Null(B)$,
proving the second conclusion.

\end{proof}

\begin{example}
$A=\left(\begin{array}{llc}1&0&0\\0&0&0\\0&0&0\end{array}\right)$ is
$psd$ and
$B=\left(\begin{array}{llc}0&b&0\\b&0&0\\0&0&1\end{array}\right)$ is
$psd$ on $Null(A)$, but whenever $b\neq 0$ and $\epsilon\neq0$
$A+\epsilon B$ is not $psd$.

This phenomenon persists when the $B$ are specialized to represent
\textit{changes} $\Delta$.  Consider the coercive $sos$ in
noncoercive representation $(x^2+y^2)^2+z^4+y^2z^2+x^2z^2$, i.e.
with Gram matrix
$\mathcal{A}=(\mathbf{E}_{200}+\mathbf{E}_{020})^{\otimes
2}+\mathbf{E}_{002}^{\otimes 2}+\mathbf{E}_{011}^{\otimes
2}+\mathbf{E}_{101}^{\otimes 2}$.  Then
$\Delta=\mathbf{E}_{002}\otimes_s\mathbf{E}_{110}-\mathbf{E}_{011}\otimes_s\mathbf{E}_{101}$
is trivially $psd$ on $Null(\mathcal{A})$, but
$\mathcal{A}+\epsilon\Delta$ is not $psd$  unless $\epsilon=0$. Here
$\Delta\cdot\mathbf{E}_{110}\mathbf{E}_{110}=0$ while
$\Delta\mathbf{E}_{110}=\frac{1}{2}\mathbf{E}_{002}$.
\end{example}

\section{Proof of Theorem \ref{thm1}}\mylabel{proof1}

Theorem \ref{thm1} follows  from the next theorem.

\begin{thm}
Given  $\gamma$, $0<\gamma<1/3$, the positive definite quartic form
of  $\mathbb{R}[u,v,w,x,y,z]$
\begin{equation}\mylabel{noncoerciverep}
f=(u^2+v^2+vw)^2+(w^2-\gamma
(x^2+y^2+z^2))^2+(wx-yz)^2+(wy-zx)^2+(wz-xy)^2
\end{equation}
is a noncoercive sum of squares.
\end{thm}

\begin{proof}
The last four terms sum to  a $pd$ form over $\mathbb{R}^4$ as shown
in the last section.  From this, positive definiteness over
$\mathbb{R}^6$ follows.  On the other hand $(1,i,0,0,0,0)\in
\mathbb{C}^6$ is a root for each of the five squared quadratics,
i.e. the real and imaginary parts of
\begin{equation}\mylabel{complexzero}
(\mathbf{e}^1+i\mathbf{e}^2)^{\otimes
2}=\mathbf{E}_{200000}-\mathbf{E}_{020000}+2i\mathbf{E}_{110000}:=\mathbf{r}+i\mathbf{q}
\end{equation}
are in the null space of the Gram matrix $\mathcal{G}_0$ that gives
representation \eqref{noncoerciverep} for $f$.  Using
\eqref{strategy}, noncoerciveness of $f$ will be proved by showing
that \textit{every} Gram matrix for $f$ contains $\mathbf{r}$ and
$\mathbf{q}$ \eqref{complexzero} in its null space.

Denote
$\Delta_1=-\frac{1}{2}\mathbf{E}_{200000}\otimes_s\mathbf{E}_{020000}+\mathbf{E}_{110000}^{\otimes
2}$.  Then \begin{equation}\mylabel{rq1}\Delta_1 \cdot
\mathbf{r}\mathbf{r}=\Delta_1 \cdot
\mathbf{q}\mathbf{q}=1\end{equation}

There is a basis
\begin{equation}\mylabel{rqbasis}
\{\Delta_1,\Delta_2,\ldots,\Delta_{105}\}
\end{equation}
for $A^{2,2,6}$ with $\Delta_1$ \eqref{rq1} as its first member so
that
\begin{equation}\mylabel{stvanish}
\Delta_j \cdot \mathbf{r}\mathbf{r}=\Delta_j \cdot
\mathbf{q}\mathbf{q}=0
\end{equation}
for all $j=2,3,\ldots,105$.  This follows because  the basis
elements of \eqref{samespan}
$\mathbf{E}_\alpha\otimes_s\mathbf{E}_\beta-\mathbf{E}_{\alpha'}\otimes_s\mathbf{E}_{\beta'}$,
$\alpha+\beta=\alpha'+\beta'$,  permit one of the equalities in
\eqref{stvanish} \textit{not} to hold only when either both
$\mathbf{E}_\alpha$ and $\mathbf{E}_\beta$ are contained in
$\{\mathbf{E}_{200000},\mathbf{E}_{020000},\mathbf{E}_{110000}\}$ or
both $ \mathbf{E}_{\alpha'}$ and $\mathbf{E}_{\beta'}$ are
contained.  The only basis element like this is $\pm \Delta_1$.

\begin{rem}\mylabel{theremark}
This relationship between a $\mathbf{z}^{\otimes 2}$,
$\mathbf{z}\in\mathbb{C}^n$, and some basis for $A^{2,2,n}$ is
general.  The uniqueness does not quite hold in $A^{2,p,n}$, $p\geq
3$, however.  For example, both $\mathbf{E}_{12}^{\otimes
2}-\frac{1}{2}\mathbf{E}_{21}\otimes_s\mathbf{E}_{03}$ and
$\mathbf{E}_{21}^{\otimes
2}-\frac{1}{2}\mathbf{E}_{12}\otimes_s\mathbf{E}_{30}$ are nonzero
as quadratic forms on the real and imaginary parts of
$(\mathbf{e}^1+i\mathbf{e}^2)^{\otimes 3}$.
\end{rem}

If $\Delta_1$ is removed from the basis \eqref{rqbasis} and $\Delta$
is taken in the subsequent span so that $\mathcal{G}_0+\Delta$ is a
Gram matrix, Proposition \ref{matrixprop} and \eqref{stvanish} then
imply that $\mathbf{r}$ and $\mathbf{q}$ will also be in the null
space of $\mathcal{G}_0+\Delta$ .  Together with \eqref{rq1} this
implies
\begin{multline}\mylabel{thisimplies}
\mbox{\textit{Any linear combination} $\Delta$ \textit{of basis
elements \eqref{rqbasis}, for which} $\mathcal{G}_0+\Delta$
\textit{is a Gram matrix }}\\\mbox{\textit{and for which at least
one of} $\mathbf{r}$ \textit{or} $\mathbf{q}$ \textit{is} not
\textit{in the null space of} $\mathcal{G}_0+\Delta$, \textit{must
have a}\hskip1in}\\\mbox{positive \textit{coefficient on}
$\Delta_1$.\hskip6in}
\end{multline}
Hence let $\delta>0$ and consider the following principal submatrix
of $\mathcal{G}_0+2\delta\Delta_1$ where the order
$\mathbf{E}_{011000}\prec\mathbf{E}_{200000}\prec\mathbf{E}_{020000}\prec\mathbf{E}_{002000}\prec\mathbf{E}_{000200}\prec
\mathbf{E}_{110000}\prec\mathbf{E}_{101000}\prec\mathbf{E}_{010100}\prec\mathbf{E}_{100100}$
(i.e. $vw\prec u^2\prec v^2\prec w^2\prec x^2\prec uv\prec uw\prec
vx\prec ux$) has been chosen, and $a=b=c=d=e=0$.  Blank entries are
\textit{zero}.

\begin{equation}\mylabel{9by9}\left(\begin{array}{ccccccccc}1+2a&1-c&1&&&&&&\\\\1-c&1
&1-\delta &-d&e&&&&\\\\1&1-\delta &1 &-a&b&&&&\\\\&-d&-a&1 &-\gamma
&&&&\\\\&e&b&-\gamma &\gamma^2&&&&
\\&& &&&2\delta &c&&\\&&&&&c&2d&&
\\&&&& &&&-2b&\\&&&&&&&&-2e
\end{array}\right)\end{equation}

The following notation for principal submatrices of \eqref{9by9}
will be used.  $[1\;3]$ denotes the submatrix
$\left(\begin{array}{ll}1+2a&1\\1&1\end{array}\right)$ formed from
the $1$st and $3$rd rows and columns  of \eqref{9by9}, etc.

The parameters $a,b,c,d,e$ correspond to the \textit{changes}\;
$2\mathbf{E}_{011000}^{\otimes
2}-\mathbf{E}_{020000}\otimes_s\mathbf{E}_{002000}$,\hskip.4in$\mathbf{E}_{020000}\otimes_s\mathbf{E}_{000200}-2\mathbf{E}_{010100}^{\otimes
2}$,\;$\mathbf{E}_{110000}\otimes_s\mathbf{E}_{101000}-\mathbf{E}_{200000}\otimes_s\mathbf{E}_{011000}$,\;
$2\mathbf{E}_{101000}^{\otimes
2}-\mathbf{E}_{200000}\otimes_s\mathbf{E}_{002000}$,\;$\mathbf{E}_{200000}\otimes_s\mathbf{E}_{000200}-2\mathbf{E}_{100100}^{\otimes
2}$ respectively.

No other \textit{nonzero} entries may be altered:  The four $1$'s in
$[1\;2\;3]$ because there is no basis element of $A^{2,2,6}$ that is
expressed using these positions.  The three entries with $\delta$
because the only change possible has already been chosen.  $[4\;5]$
because the quartic form $f(0,0,w,x,y,z)$ has a unique Gram matrix
by Theorem \ref{gammaquartic}, and $[4\;5]$ is a submatrix of that
Gram matrix; if a $\mathcal{G}_0+\Delta$ is a Gram matrix, then by
deleting all rows and columns that involve the variables $u$ and $v$
one obtains a Gram matrix for $f(0,0,w,x,y,z)$.

When $a=c=0$ it follows that $\det[1\;2\;3]=-\delta^2<0$.  Since all
principal minors of a $psd$ matrix must be nonnegative, $a=c=0$
cannot hold.  It will first be shown that $a=0$ is necessary and
then that $c=\delta$ is necessary, leading to a  contradiction that
proves the theorem .

The determinant of $[1\;3]$ forces $a\geq 0$.  Introducing $b$,
$\det[3\;4\;5]=-(b-a\gamma)^2$ whence $b=a\gamma\geq 0$.  But
submatrix $[8]$ implies $b\leq0$ whence $a=0$ also.

With $a=0$ it follows that $\det[1\;2\;3]=-(\delta-c)^2$ whence
$c=\delta$.  Consequently $[6\;7]$ requires $d>0$.  Now
$\det[2\;4\;5]=-(e-d\gamma)^2$ whence $e>0$ contradicting submatrix
$[9]$.

\end{proof}

\section{A $6$th order example}\mylabel{cubic}

Consider the family of sextics
\begin{equation}
f_\rho (x,y,z)=x^2(\rho^2 x^2 +\rho y^2 -\frac{1}{2} z^2 )^2
+y^2(\rho^2 y^2 +\rho z^2 -\frac{1}{2} x^2 )^2+z^2(\rho^2 z^2 +\rho
x^2 -\frac{1}{2} y^2 )^2
\end{equation}
The three cubic polynomials that are squared have a common
nontrivial root only when $\rho=0, \rho^3=-\frac{1}{2},
\rho^3=-\frac{5+3\sqrt{3}}{4} \mbox{ or }
\rho^3=\frac{-5+3\sqrt{3}}{4}$.  In each case the root can be taken
 in $\mathbb{R}^3$.  Thus $f_\rho$ is $pd$ if and only if $\rho^3$
does not take the four listed values.  In addition, every $pd$ form
$f_\rho$ is  coercive.

Put $\eta_0=(1+\sqrt{5})^{-3}$.  Then for the Choi-Lam sextics
\eqref{qeta},  $s_{\eta_0}=(1+\sqrt{5})f_\rho$ when
$\rho=(1+\sqrt{5})^{-1}$.  It will be shown that $(1+\sqrt{5})^{-1}$
belongs to an interval of $\rho$'s for which the $f_\rho$ have
unique Gram matrices.  This uniqueness implies, as in the quartic
case, that $\eta_0$ is the smallest value of $\eta$ for which
$s_\eta$ is an $sos$.  The identity used in \eqref{2different} may
be replaced with $x^6+y^6=(x^3-2xy^2)^2+(y^3-2x^2y)^2$.

Hence, an apparent Gram matrix $\mathcal{G}_\rho$ for each $f_\rho$
is
$$(\rho^2 \mathbf{E}_{300}+\rho \mathbf{E}_{120} -\frac{1}{2}
\mathbf{E}_{102} )^{\otimes 2}+(\rho^2 \mathbf{E}_{030}+\rho
\mathbf{E}_{012} -\frac{1}{2} \mathbf{E}_{210} )^{\otimes 2}+(\rho^2
\mathbf{E}_{003}+\rho \mathbf{E}_{201} -\frac{1}{2} \mathbf{E}_{021}
)^{\otimes 2}$$ acting on the space $S^3(\mathbb{R}^3)$ which has
$10$ dimensions.  Therefore using
$\mathbf{E}_\alpha\cdot\mathbf{E}_\alpha=\frac{\alpha!}{6}$  the
null space for $\mathcal{G}_\rho$ is spanned by the vectors
$\mathbf{E}_{300}-3\rho \mathbf{E}_{120}, \mathbf{E}_{120}+2\rho
\mathbf{E}_{102}, \mathbf{E}_{030}-3\rho \mathbf{E}_{012},
\mathbf{E}_{012}+2\rho \mathbf{E}_{210}, \mathbf{E}_{003}-3\rho
\mathbf{E}_{201}, \mathbf{E}_{201}+2\rho \mathbf{E}_{021}\mbox{ and
} \mathbf{E}_{111}$.  A general linear combination is
\begin{multline}
\mathbf{g}=a\mathbf{E}_{300}+3(b-\rho a )\mathbf{E}_{120} + 6\rho b
\mathbf{E}_{102}+c\mathbf{E}_{030}+3(d-\rho c )\mathbf{E}_{012} +
6\rho d \mathbf{E}_{210}\\+e\mathbf{E}_{003}+3(f-\rho e
)\mathbf{E}_{201} + 6\rho f \mathbf{E}_{021}+6g\mathbf{E}_{111}
\end{multline}

The $27$ dimensions of the subspace $A^{2,3,3}$ of \textit{changes}
may be briefly described as follows.

$$\frac{1}{2}\mathbf{E}_{300}\otimes_s\mathbf{E}_{120}-\frac{1}{2}\mathbf{E}_{210}\otimes_s\mathbf{E}_{210}$$
is representative of $6$ changes.

$$\frac{1}{2}\mathbf{E}_{300}\otimes_s\mathbf{E}_{111}-\frac{1}{2}\mathbf{E}_{210}\otimes_s\mathbf{E}_{201}$$
is representative of $3$.

$$\frac{1}{2}\mathbf{E}_{300}\otimes_s\mathbf{E}_{030}-\frac{1}{2}\mathbf{E}_{210}\otimes_s\mathbf{E}_{120}$$
 representative of $3$.

$$\frac{1}{2}\mathbf{E}_{300}\otimes_s\mathbf{E}_{021}-\frac{1}{2}\mathbf{E}_{201}\otimes_s\mathbf{E}_{120}$$
 representative of $6$.

 $$\frac{1}{2}\mathbf{E}_{300}\otimes_s\mathbf{E}_{021}-\frac{1}{2}\mathbf{E}_{210}\otimes_s\mathbf{E}_{111}$$
 representative of $6$.

 $$\frac{1}{2}\mathbf{E}_{210}\otimes_s\mathbf{E}_{012}-\frac{1}{2}\mathbf{E}_{111}\otimes_s\mathbf{E}_{111}$$
 representative of $3$.
 Keeping in mind the examples
 $\mathbf{E}_{300}\cdot\mathbf{E}_{300}=1,
 \mathbf{E}_{120}\cdot\mathbf{E}_{120}=1/3 \mbox{ and }
 \mathbf{E}_{111}\cdot\mathbf{E}_{111}=1/6$, and computing  $\Delta \cdot \mathbf{g}\mathbf{g}$ for each change yields the quadratic polynomials

$-\rho a^2 + ab-4\rho^2 d^2$

$-\rho c^2 + cd-4\rho^2 f^2$

$-\rho e^2 + ef-4\rho^2 b^2$

$2\rho ab- \rho^2 e^2 -f^2 +2\rho ef$

$2\rho cd- \rho^2 a^2 -b^2 +2\rho ab$

$2\rho ef- \rho^2 c^2 -d^2 +2\rho cd$\vskip.1in

$ag-2\rho df+2\rho^2 de$

$cg-2\rho bf+2\rho^2 af$

$eg-2\rho bd+2\rho^2 bc$\vskip.1in

$ac+2\rho^2 ad -2 \rho bd$

$ce+2\rho^2 cf -2 \rho df$

$ae+2\rho^2 be -2 \rho bf$\vskip.1in

$3\rho af-\rho^2 ae-bf +\rho be$

$3\rho bc-\rho^2 ac-bd +\rho ad$

$3\rho ed-\rho^2 ce-df +\rho cf$

$-\rho ac +ad -4\rho^2 bd$

$-\rho ce +cf -4\rho^2 df$

$-\rho ae +be -4\rho^2 bf$\vskip.1in

$af-dg$

$bc-fg$

$de-bg$

$-\rho ac +ad +\rho eg -fg$

$-\rho ce +cf +\rho ag -bg$

$-\rho ae +be +\rho cg -dg$\vskip.1in

$-2\rho^2 cd +2 \rho d^2 -g^2$

$-2\rho^2 ef +2 \rho f^2 -g^2$

$-2\rho^2 ab +2 \rho b^2 -g^2$\vskip.1in

Linearly order the $28$ quadratic monomials
$a^2,b^2,\ldots,g^2,ab,ac,\ldots,ag,bc,\\\ldots,eg,fg$ as before and
put the resulting $27\times 28$ coefficient matrix into reduced
echelon form.  When the $26$th column (the $ef$ column) is removed
the result is the \textit{identity} matrix.  Putting $\sigma=
\frac{1-16\rho^3}{\rho (1-4\rho^3)}, \tau=\frac{3\rho}{1-4\rho^3}
\mbox{ and } \phi=\frac{4\rho^2(2\rho^3+1)}{1-4\rho^3}$, the $26$th
column has successive entries $$-\sigma, -\tau, -\sigma,
-\tau,-\sigma,
-\tau,-\phi,-1,0,0,0,0,0,0,0,0,0,0,-1,0,0,0,0,0,0,0,0$$

Thus an equivalent set of polynomials is
\begin{multline}\mylabel{equivpolys2}
$$\hskip1.82in a^2-\sigma ef$$\\$$b^2-\tau ef$$\\
$$c^2-\sigma ef$$\\$$\hskip2in d^2-\tau ef\hskip2in$$\\$$e^2-\sigma ef$$\\$$\hskip2.1in f^2-\tau ef\hskip2.1in$$\\$$\hskip2in g^2-\phi ef\hskip2.1in$$
\\$$\hskip2in ab- ef\hskip2.2in$$\\$$\hskip1.8in cd- ef\hskip2.2in$$
\end{multline}
together with the remaining $18$ indefinite monomials none of which
appear in the polynomials \eqref{equivpolys2}.  By
\eqref{uniqueness?} a sufficient requirement for $f_\rho$ to have a
unique Gram matrix is that there exists no nontrivial linear
combination of the polynomials \eqref{equivpolys2} that is a
definite or  semi-definite quadratic polynomial in the variables
$a,\ldots,g$.  This requirement is equivalent to showing for a given
$\rho$ that every nontrivial choice of parameters
$A,B,C,D,E,F,G,J,K$ in

\begin{equation}\mylabel{7by7}\tiny\left(\begin{array}{ccccccc}2A&J\!\!-\!\!A\sigma\!\!-\!\!B\tau&&&&&\\\\J\!\!-\!\!A\sigma\!\!-\!\!B\tau&2B
&&&&&\\\\&&2C&K\!\!-\!\!C\sigma\!\!-\!\!D\tau&&&
\\\\&&K\!\!-\!\!C\sigma\!\!-\!\!D\tau&2D&&&\\\\&&&&2E&\!-\!E\sigma\!\!-\!\!F\tau\!\!-\!\!G\phi\!\!-\!\!J\!\!-\!\!K&
\\\\&&&&\!-\!E\sigma\!\!-\!\!F\tau\!\!-\!\!G\phi\!\!-\!\!J\!\!-\!\!K&2F&\\\\&&&&&&2G
\end{array}\right)\normalsize\end{equation}
produces an indefinite matrix.

When $\sigma, \tau \mbox{ and } \phi$ are not all of the same sign
there exist, by \eqref{equivpolys2}, choices of positive
$A,\ldots,G$ that make \eqref{7by7} definite.  Lack of a common sign
holds for $-1/2<\rho^3<0$ and $ 1/16 \leq \rho^3$ .  When
$\rho^3<-1/2$ each of $\sigma, \tau \mbox{ and } \phi$ is negative
while each is positive for $0<\rho^3<1/16$.

Restricting to those nontrivial choices with $G=J=K=0$, all produce
indefinite matrices \eqref{7by7} if and only if $\sigma\tau>1$.  For
example, the $2\times 2$ minor $4EF-(E\sigma+F\tau)^2<0$ if and only
if $\sigma\tau>1$ if and only if
$-\frac{5+3\sqrt{3}}{4}<\rho^3<\frac{-5+3\sqrt{3}}{4}$.  Therefore
the remaining intervals for $\rho^3$ for which all nontrivial
\eqref{7by7} are possibly not definite are the open intervals
$(-\frac{5+3\sqrt{3}}{4},\frac{-1}{2}) \mbox{ and }
(0,\frac{-5+3\sqrt{3}}{4})$.  That $\phi$ shares the same sign with
$\sigma$ and $\tau$ in these intervals shows that choosing $G>0$
does not restrict these intervals further.  Neither can nonzero
choices of $J$ and $K$.  The endpoints of the intervals yield
$f_\rho$ that are not $pd$.

The foregoing proves

\begin{thm}\mylabel{symevenunique}
The forms $f_\rho$ are $pd$ and have unique Gram matrices if and
only if  $-\frac{5+3\sqrt{3}}{4}< \rho^3 <-\frac{1}{2} \mbox{ or  }
0 <\rho^3 < \frac{-5+3\sqrt{3}}{4}$.  All other $pd$ $f_\rho$ have
Gram matrices of rank 10.  Each $f_\rho$, for $\rho^3$ not equal to
the endpoints of the above intervals, is coercive.
\end{thm}

$(\sqrt{5}+1)^{-3}$ is contained in the second interval.

To prove Theorem \ref{thm2} we will be content with a single
example.  Take $\rho=-1$.

\begin{thm}
The positive definite sextic form of $\mathbb{R}[w,x,y,z]$
\begin{multline}\mylabel{symeven} g(w,x,y,z):=f_{-1}(\sqrt{w^2+x^2},y,z)=(w^3+wx^2-wy^2-\frac{1}{2}wz^2)^2\\+(xw^2+x^3-xy^2-\frac{1}{2}xz^2)^2
+(y^3-yz^2-\frac{1}{2}yw^2-\frac{1}{2}yx^2)^2
+(z^3-zw^2-zx^2-\frac{1}{2}zy^2)^2\end{multline} is a noncoercive
sum of squares.
\end{thm}

\begin{proof}
Let $\mathcal{G}_0$ denote the apparent Gram matrix for $g$ and let
$\mathcal{F}_{-1}$ denote the unique Gram matrix for $f_{-1}$.

For $\mathbf{z}\in \mathbb{C}^4$ denote $z_1^2+z_2^2=\xi^2$.  Then
the precise relationship between common complex roots for $sos$
representations of $g$ and $f_{-1}$ is
$\mathcal{G}_0\mathbf{z}^{\otimes 3}=\mathbf{0}$ if and only if
$\mathcal{F}_{-1}(\xi,z_3,z_4)^{\otimes 3}=\mathbf{0}$. Consequently
by Theorem \ref{symevenunique} and \eqref{strategy} $\xi=z_3=z_4=0$
when $\mathbf{z}^{\otimes 3}$ is in the null space of
$\mathcal{G}_0$.  Thus $\mathbf{z}:=(1,i,0,0)$ may be taken, up to
scaling, as the only nontrivial common root in the $sos$
representation \eqref{symeven} for $g$.

For $g$ to be coercive there must exist a $\Delta$ such that
$\mathcal{G}_0 + \Delta$ is a Gram matrix and $\Delta
\mathbf{z}^{\otimes 3}\neq \mathbf{0}$ \eqref{strategy}.  Therefore,
similarly to the quartic case, at least one of
$\Delta_1=\mathbf{E}_{1200}^{\otimes
2}-\frac{1}{2}\mathbf{E}_{2100}\otimes_s\mathbf{E}_{0300}$ or
$\Delta_2=\mathbf{E}_{2100}^{\otimes
2}-\frac{1}{2}\mathbf{E}_{1200}\otimes_s\mathbf{E}_{3000}$ (see
Remark \ref{theremark}) must be included in $\Delta$ with a
\textit{positive} coefficient.  However, if $\Delta'$ is obtained
from $\Delta$ by permuting the 1st and 2nd components of each
multi-index of the basis elements \eqref{E}, then $\mathcal{G}_0 +
\Delta'$ would also be a Gram matrix because of the symmetry in $w$
and $x$ of \eqref{symeven}.  Further, because of  positive
semi-definiteness, $\mathbf{z}^{\otimes 3}$ is not in the null space
of $\mathcal{G}_0 +\frac{1}{2}\Delta+ \frac{1}{2}\Delta'$ when it is
not in the null space of $\mathcal{G}_0 + \Delta$.  Consequently,
for $g$ to be coercive, values for the parameters $a,b,c,d$ with
$a+b+c+d=0$ in
\begin{equation}\mylabel{new10by10}\left(\begin{array}{rrrrrrrccc}1+2\delta &1-\delta
&-1+a&&&&&&&\\1-\delta&1&-1&&&&&&&\\-1+a&-1&1&&&&&&&\\&&&1+2\delta
&1-\delta &-1+b&&&&\\&&&1-\delta&1&-1&&&&\\&&&-1+b&-1&1
&&&&\\&&&&&&1&-\frac{1}{2}&-\frac{1}{2}&
\\\\&&&&&&-\frac{1}{2}&\frac{1}{4}&\frac{1}{4}+c&\\\\&&&&&&-\frac{1}{2}&\frac{1}{4}+c&\frac{1}{4}&
\\\\&&&&&&&&&d\end{array}\right)\end{equation}
 must be
found that make \eqref{new10by10} $psd$ when $\delta>0$.  Here
\eqref{new10by10} is the principal submatrix of $\mathcal{G}_0
+2\delta\Delta_1+ 2\delta\Delta_2$ corresponding to
$\mathbf{E}_{1200}\prec\mathbf{E}_{3000}\prec\mathbf{E}_{1020}\prec\mathbf{E}_{2100}\prec\mathbf{E}_{0300}\prec
\mathbf{E}_{0120}\prec\mathbf{E}_{0030}\prec\mathbf{E}_{2010}\prec\mathbf{E}_{0210}\prec\mathbf{E}_{1110}$
(i.e. $wx^2\prec w^3\prec wy^2\prec w^2x\prec x^3\prec xy^2\prec
y^3\prec w^2y\prec x^2y\prec wxy$).  The parameters represent the
three  \textit{changes}
$\mathbf{E}_{2100}\otimes_s\mathbf{E}_{0120}-\mathbf{E}_{1110}\otimes_s\mathbf{E}_{1110},
\mathbf{E}_{2010}\otimes_s\mathbf{E}_{0210}-\mathbf{E}_{1110}\otimes_s\mathbf{E}_{1110},
\mbox{ and }
\mathbf{E}_{1200}\otimes_s\mathbf{E}_{1020}-\mathbf{E}_{1110}\otimes_s\mathbf{E}_{1110}$.

With the same notation as in the quartic case, the submatrix $[2\;
3]$ of \eqref{new10by10} is \textit{fixed} because it is a submatrix
of the unique Gram matrix $\mathcal{F}_{-1}$ for
$g(w,0,y,z)=f_{-1}(w,y,z)$.  So is $[5\;6]$ because
$g(0,x,y,z)=f_{-1}(x,y,z)$.  In the same way $[7\;8]$ and $[7\;9]$
are fixed.  With no other choices and $\det[1\;2\;3]=-(a-\delta)^2$
it follows that $a=\delta$ is forced.  In the same way $b=\delta$
and $c=0$.  Thus $d=-2\delta$, a contradiction, and $g$ cannot be
coercive.

\end{proof}

\begin{rem}
By Nullstellens\"atze (see pp. 56-57 of \cite{Pfi95}) every
collection of homogeneous polynomials $p_1,\ldots,p_r \in
\mathbb{C}[x_1,\ldots,x_n]$ with $1\leq r <n$ has a common
nontrivial zero $\mathbf{a}\in \mathbb{C}^n$ to the system of
equations $p_1=\cdots=p_r=0$ while the corresponding statement, for
the polynomial ring $\mathbb{R}[x_1,\ldots,x_n]$ and $\mathbb{R}^n$
in place of $\mathbb{C}^n$, holds only when all of the degrees
$d_1,\ldots,d_r$ of the polynomials $p_1,\ldots,p_r$ are not
\textit{even}.  Thus the sextic example here is required to be the
sum of at least $4$ squares in order to be $pd$ while the quartic
examples are $pd$ with but $5$ squares of quadratics in the $6$
indeterminates.  The $5$ quadratics necessarily share a nontrivial
complex root while the $4$ cubics need not, though they do.
\end{rem}
\section{The game}\mylabel{game}

Starting with the collection of $pd$ $sos$ in $P_{3,4}$ one can
obtain the coercive result \eqref{positiveresult} without using
Hilbert's theorem on ternary quartics by considering several generic
cases. One shows that the ranks of the Gram matrices arising in each
case can be built up by adding changes $\Delta$ as delineated in
Proposition \ref{matrixprop}.  When attempting to show that the $pd$
elements of $\Sigma_{4,4}$ are coercive the number of cases is
significantly higher.

The vector space $S^2(\mathbb{R}^n)$ is isomorphic to the space of
real symmetric $n\times n$ matrices by assigning $\mathbf{t}\in
S^2(\mathbb{R}^n)$ to the matrix with the coordinates
$\mathbf{t}\cdot \mathbf{e}^i \mathbf{e}^j$ as entries $1\leq
i,j\leq n$.  Every change $\Delta \in A^{2,2,n}$ yields a quadratic
form $\Delta\cdot\mathbf{t}\mathbf{t}$ that is a linear combination
of the $2\times 2$ minors of the symmetric matrix $\mathbf{t}$.  The
argument of Section 3 that can show that  a $pd$ quartic with Gram
matrix $\mathcal{G}$ has $\mathcal{G}$ as its unique Gram matrix
amounts to showing that a general matrix in
$Null(\mathcal{G})\subset S^2(\mathbb{R}^n)$ has the property that
every nontrivial linear combination of its $2\times 2$ minors is
indefinite.  For example, the $pd$ quartic
$$f=(x_1^2+x_2^2-x_3^2-x_4^2)^2+(2x_2x_3-x_3^2+x_4^2)^2+(x_1x_3-x_2x_4)^2+(x_1-x_3)^2x_4^2$$
has the basis $2\mathbf{E}_{1100},
\mathbf{E}_{2000}-\mathbf{E}_{0200},
\mathbf{E}_{2000}+\mathbf{E}_{0200}+\mathbf{E}_{0020}+\mathbf{E}_{0002},
\mathbf{E}_{0020}-\mathbf{E}_{0002}+2\mathbf{E}_{0110},
2\mathbf{E}_{1010}+2\mathbf{E}_{0101},
2\mathbf{E}_{1001}+2\mathbf{E}_{0011}$ for the null space of its
apparent Gram matrix.  The first two basis elements are the
imaginary and real parts of $\mathbf{z}\otimes\mathbf{z}$ for
$\mathbf{z}=(1,i,0,0)$ the common complex root for the $sos$ $f$.  A
general linear combination of the basis elements corresponds to the
$4\times 4$ matrix
\begin{equation}\mylabel{gamematrix}\mathbf{t}=\left(\begin{array}{cccc}b+c&a&e&f\\\\a&-b+c&d&e\\\\e&d&c+d&f
\\\\f&e&f&c-d
\end{array}\right)\end{equation}

Here, however, there is a nontrivial linear combination of the
$2\times 2$ minors that is not indefinite.  Otherwise $f$ would
provide a noncoercive example for $n=4$.  To prove that $f$ is
coercive it is necessary to produce a linear combination of minors
of the form
\begin{equation}\mylabel{minorscomb}
a^2+b^2-c^2+\Delta\cdot\mathbf{t}\mathbf{t}
\end{equation}
that is $psd$ and where the last term does not include the principle
$2\times 2$ minor $\det[1\;2]$.  This is the same observation as
\eqref{thisimplies}.  It might not be clear that the last term can
be made up of minors that yield a positive coefficient on the
monomial $c^2$ without introducing more indefiniteness. However, it
can be done.  To express $c^2$ itself as a linear combination of the
remaining 19 independent $2\times 2$ minors it is necessary to use
18 of them.  In fact, \eqref{minorscomb} can be made $pd$ and thus
$f$ possesses a Gram matrix of full rank by Proposition
\ref{matrixprop}.

By a linear change of variables in $\mathbb{R}^n$ any nontrivial
common complex root for a $pd$ quartic $sos$ may be taken to be
$(1,i,0,0,\ldots)$.  Therefore the precise setup of the principal
submatrix $[1\;2]$ of \eqref{gamematrix}, together with the presence
in some way of the variable $c$ outside $[1\;2]$, is a typical setup
for the null spaces of Gram matrices when trying to answer question
\eqref{question} in the quartic cases.  When $c$ does not occur
outside $[1\;2]$ real values may be assigned to the variables making
$[1\;2]$ and  $\mathbf{t}$  rank-1 matrices, contradicting the
positive definiteness of the form $f$.  When only $a$ and $b$ occur
in $[1\;2]$ $f$ can be written as a $sos$ that includes the term
$(x_1^2+x_2^2)^2=x_1^4+2x_1^2x_2^2+x_2^4$ in the sum.

These observations lead to the following diversion.\vskip.1in
\noindent 1. Set up the principal submatrix $[1\;2]$ of an $n\times
n$ symmetric matrix $\mathbf{t}$ exactly as in
\eqref{gamematrix}.\vskip.1in \noindent 2.  Write linear
combinations of $c$ and a number of other real variables for the
remaining entries. Variable $c$ must be used while $a$ and $b$ may
not.\vskip.1in \noindent 3.  The choices made in Step 2 are not
allowed to result in a rank-1 matrix for any choice of real variable
values.  This can usually be checked by inspecting for zeros an
$sos$ quartic form, i.e. Gram matrix,  which will have the $n\times
n$ matrix as its null space.\vskip.1in \noindent 4.  Search for a
linear combination of $2\times 2$ minors (not including
$\det[1\;2]$) which when added to $a^2+b^2-c^2$ results in a $psd$
quadratic form.\vskip.1in

When $n=4$ or $5$ there are two or three ways to win this game. Find
a setup for which the goal of Step 4 cannot be achieved.  Or, when
Step 4 does result in a $psd$ quadratic but never a $pd$ quadratic,
show that the resulting change $\Delta$ always satisfies
$\Delta\mathbf{t}\neq \mathbf{0}$ for some choice of real variable
values.  See Proposition \ref{matrixprop}.  Or, prove that neither
of these outcomes is ever possible for any $\mathbf{t}$ constructed
according to Steps 1, 2 and 3, thus proving that every $pd$ $sos$ is
coercive.

\section{Final remark on coercive integro-differential forms}

The results of this article when combined with the Aronszajn-Smith
Theorem show that there exist homogeneous constant coefficient
elliptic operators $L$ with formally positive integro-differential
forms \eqref{intform} for which a coercive estimate like
\eqref{pdecoercive} is never true.  However, such an $L$ could have
an integro-differential form like \eqref{intdiffform} which is not
formally positive but which satisfies the coercive estimate
\eqref{pdecoercive} when \eqref{intdiffform} is used on the left
side in place of \eqref{intform}.  The author claims this to be
always true in the quartic, i.e. 4th order operator, cases.  The
proof necessarily uses Agmon's characterization of coerciveness and
will appear elsewhere.  Thus Agmon's characterization is needed in
order to answer the coerciveness problem for differential operators
even when those operators possess formally positive
integro-differential forms.

\bibliographystyle{amsalpha}
\bibliography{4thOrder}

\end{document}